\RequirePackage{fix-cm}
\documentclass[smallextended]{svjour3}       
\smartqed  
\usepackage{graphicx}
\usepackage{amsmath,amssymb,amsfonts,amscd} 
\usepackage{enumerate}
\usepackage{url}         
\usepackage{colonequals} 
\usepackage{hyperref}

\hypersetup{
    bookmarks=true,         
    unicode=false,          
    pdftoolbar=true,        
    pdfmenubar=true,        
    pdffitwindow=false,     
    pdfstartview={FitH},    
    pdftitle={My title},    
    pdfauthor={Author},     
    pdfsubject={Subject},   
    pdfcreator={Creator},   
    pdfproducer={Producer}, 
    pdfkeywords={keywords}, 
    pdfnewwindow=true,      
    colorlinks=true,       
    linkcolor=blue,          
    citecolor=blue,        
    filecolor=magenta,      
    urlcolor=cyan           
}

\usepackage{color}

\def\B{\mathcal{ B}}
\def\R{\mathbf{ R}}

\def\P{\mathcal{P}}

\long\def\red#1{{\color{red}#1}}

\begin{document}

\title{Analysis of the expected density of internal equilibria in random evolutionary multi-player multi-strategy games
}


\author{Manh Hong Duong         \and
        The Anh Han 
}


\institute{Manh Hong Duong \at
              Mathematics Institute, University of Warwick, UK. \\
              \email{m.h.duong@warwick.ac.uk}           
           \and
           The Anh Han \at
              School of Computing, Teesside University, UK.\\
              \email{T.Han@tees.ac.uk}
}
\date{Received: date / Accepted: date}

\maketitle

\begin{abstract}
In this paper, we study the distribution and behaviour of internal equilibria in a $d$-player $n$-strategy random evolutionary game where the game payoff matrix is generated from normal distributions.
The study of this paper reveals and exploits interesting connections between evolutionary game theory and random polynomial theory. The main contributions  of the paper are some qualitative and quantitative results on the expected density, $f_{n,d}$, and the expected number, $E(n,d)$, of (stable) internal equilibria. Firstly, we show that in multi-player two-strategy games, they behave asymptotically as $\sqrt{d-1}$ as $d$ is sufficiently large. Secondly, we prove that they are monotone functions of $d$. We also make a conjecture for games with more than two strategies. Thirdly, we provide numerical simulations for our analytical results and to support the conjecture. As consequences of our analysis, some qualitative and quantitative results on the distribution of zeros of a random Bernstein polynomial are also obtained.

%

\end{abstract}
\section{Introduction}
\subsection{Motivation}
Evolutionary game theory (EGT) has been proven to be a suitable mathematical framework to model biological and social  evolution   whenever the success of an individual depends on the presence or absence of other strategies  \cite{maynard-smith:1982to,hofbauer:1998mm,nowak:2006bo}. EGT was introduced in 1973 by Smith and Price \cite{SP73} as an  application of classical game theory to biological contexts,  and has since then been widely and successfully applied   to various fields, not only  biology itself, but also  ecology, population genetics, and computational   and social sciences \cite{maynard-smith:1982to,axelrod:1984yo,hofbauer:1998mm,nowak:2006bo,broom2013game,Perc2010109,HanBook2013,gokhale2014evolutionary}. In these contexts, the payoff obtained  from  game interactions is translated into reproductive fitness or social success \cite{hofbauer:1998mm,nowak:2006bo}. Those strategies that achieve higher fitness or are more successful, on average,  are favored by natural selection, thereby  increase in their frequency.  Equilibrium points  of such a dynamical system are the compositions of strategy frequencies where all the strategies have the same average fitness.  Biologically, they  predict the co-existence of different types in a population and the maintenance of polymorphism. 

As in classical game theory with the dominant concept of Nash equilibrium \cite{mclennan2005asymptotic,mclennan2005expected}, the analysis of equilibrium points in  random  evolutionary games is of great importance because it allows one to describe  various generic  properties, such as the overall complexity  of interactions and the average behaviours, in a dynamical system. Understanding properties of equilibrium points in a concrete system is important, but what if the system itself is not fixed or undefined?  Analysis of random games  is  insightful for such scenarios. To this end, it is    ambitious and desirable to answer the following general  questions:

\textbf{How are the equilibrium points distributed? How do they behave when the number of players and strategies change?}

Mathematical analysis of equilibrium points and their stability in a general (multi-player multi-strategy) evolutionary game is challenging because one would need to cope with systems of multivariate polynomial equations of high degrees (see Section \ref{sec: pre} for more details). Nevertheless,  some recent attempts,  both through  numerical and analytical approaches, have  been made. One approach is to study the probabilities of having a concrete number of equilibria, whether  all   equilibrium points or only the stable ones are counted, if the payoff entries follow a certain probability distribution \cite{gokhale:2010pn,HTG12}. This approach has the advantage that these probabilities provide   elaborate  information on the distribution of the equilibria. However, it consists of  sampling and solving of a system of  multivariate polynomial equations; hence is restricted, even when using numerical simulations, to games of a small number of  players and/or small number of strategies: it is known that it is impossible to (analytically) solve an algebraic equation of a degree  greater than $5$ \cite{able:1824aa}. Another possibility is to analyze the attainability of the patterns and the maximal number of    evolutionarily stable strategies  (ESS) \cite{maynard-smith:1982to,hofbauer:1998mm},  revealing  to some extent the complexity of the interactions.  This line of research has been  paid  much attention in  evolutionary game theory and  other biological fields such as population genetics \cite{karlin:1980aa,vickers:1988aa,haigh1989large,broom:1993pa,broom:1997aa,altenberg:2010tp,gokhale:2010pn,HTG12,gokhale2014evolutionary}.
More recently, in \cite{DH15}, the authors investigate the expected number of internal equilibria in a multi-player multi-strategy  random evolutionary game where the game payoff matrix is generated from normal distributions. By connecting  EGT and random polynomial theory, they describe a computationally implementable formula  of the mean number of  internal equilibrium  points for the general case, lower and upper bounds the  multi-player two-strategy random games, and a close-form formula for the two-player multi-strategy games.

In this paper,  we address  the aforementioned questions, i.e., of analysing  distributions and   behaviours of the internal equilibria of a random evolutionary game, in an \textit{average} manner.  More specifically, we first analyse the expected density of internal equilibrium points, $f_{n,d}(\mathbf{t})$, i.e. the expected number of  such  equilibrium points  per unit length at  point $\mathbf{t}$,  in a $d$-player $n$-strategy  random evolutionary game where the game payoff matrix is generated from a normal distribution (for short, normal evolutionary games). Here the parameter $\mathbf{t}=(t_i)_{i=1}^{n-1}$, with $t_i=\frac{x_i}{x_n}$, denotes the ratio of frequency of strategy $i\in\{1,\cdots, n-1\}$ to that of strategy $n$, respectively (more details in  Section \ref{sec: pre}). In such a random game, we then analyse the expected number  of internal equilibria, $E(n,d)$, and, as a result, characterize  the expected number  of internal \textit{stable} equilibria, $\mathit{SE}(n,d)$. We obtain both quantitative (asymptotic formula) and qualitative (monotone properties) results of $f_{n,d}$ and $E(n,d)$,  as functions of the ratios, $\mathbf{t}$, the number of players, $d$, and that of strategies, $n$. 

To obtain these results, we develop further the connection between EGT and random polynomial theory  explored in \cite{DH15}, and  more importantly,  establish appealing (previously unexplored) connections to the well-known  classes of polynomials, the  Bernstein polynomials and Legendre polynomials. In contrast to the direct approach used  in \cite{gokhale:2010pn,HTG12}, our  approach avoids sampling and solving a system of  multivariate polynomial equations, thereby  enabling us to study  games with large numbers of players and/or strategies.  

We now summarise the main results of the present paper.
\subsection{Main results}
The main analytical results of the present paper can be summarized into three categories: asymptotic behaviour of the density function and the expected number of (stable) equilibria, a connection between the density function with the Legendre polynomials, and monotonic behaviour of the density function. In addition, we  provide  numerical results and illustration for the general games when both the numbers of players and strategies are large.

To precisely describe our main results, we introduce  the following notation regarding  asymptotic behaviour of two given functions $u$ and $v\colon[0,+\infty)\rightarrow [0,+\infty)$
\begin{align*}
&u\lesssim v \quad \Leftrightarrow \quad\text{there exists a positive constant $C$ such that for all $k\in[0,+\infty)$}
\\&\hspace*{3cm} u(k)\leq C v(k),
\\&u\sim v \quad \Leftrightarrow \quad\text{it holds that}\quad u\lesssim v\quad\text{and}\quad v\lesssim u.
\end{align*}

Note that throughout the paper we sometimes put  arguments of a function  as subscripts. For instance, the expected density of internal equilibrium points, $f_{n,d}(\mathbf{t})$, besides $\mathbf{t}$,  is also analyzed  as a function of $n$ and $d$. We will explicitly state which parameter(s) is being varied whenever necessary to avoid the confusion. 

The main results of the present paper are the following. As described above,  $f_{n,d}(\mathbf{t})$ denotes the expected number of  internal  equilibrium points  per unit length at  point $\mathbf{t}$,  in a $d$-player $n$-strategy  random evolutionary game where the game payoff matrix is generated from a normal distribution;  $E(n,d)$ the  expected number  of internal equilibria;  and $\mathit{SE}(n,d)$  the expected number  of internal stable equilibria. The formal definitions of these three functions are given in Section \ref{sec: pre}.  
\begin{enumerate}[(1)]
\item In Theorem \ref{theo: concentration}, we prove the following  asymptotic behaviour  of $f_{2,d}(t)$ for all $t>0$: $f_{2,d}(t)\sim \sqrt{d-1}$. We also prove that $f_{2,d}(t)$ is always bounded from above and $\lim_{d\rightarrow\infty}\frac{f_{2,d}(t)}{d-1}=0$.
\item In Theorem \ref{theo: behavior of E2}, we prove a novel upper bound for the expected number of multi-player two-strategy random games, $E(2,d)\lesssim \sqrt{d-1}\ln(d-1)$ and obtain its limiting behaviour: $\lim\limits_{d\rightarrow\infty}\frac{\ln E(2,d)}{\ln(d-1)}=\frac{1}{2}$. This upper bound is sharper than the one obtained in  \cite[Theorem 2]{DH15}, which is, $E(2,d)\lesssim (d-1)^\frac{3}{4}$. These results  lead to two important  corollaries. First, we  obtain a sharper bound for the expected number of stable equilibria, $\mathit{SE}(2,d)\lesssim \frac{1}{2}\sqrt{d-1}\ln(d-1)$, and the corresponding  limit, $\lim\limits_{d\rightarrow\infty}\frac{\ln \mathit{SE}(2,d)}{\ln(d-1)}=\frac{1}{2}$, see Corollary \ref{cor: expected stable equi}. The second corollary, Corollary \ref{cor: expected zeros Bernstein}, is mathematically significant, in which we  obtain lower and upper bounds and a limiting behaviour of the expected number of real zeros of a random Bernstein polynomial.
\item In Theorem \ref{theo: fd interm of Pd}, we establish an expression of $f_{2,d}(t)$ in terms of the Legendre polynomial and its derivative. 
\item In Theorem \ref{theo: fd interms of Pd and Pd-1}, we express $f_{2,d}(t)$ in terms of the Legendre polynomials of two consecutive order.
\item In Theorem \ref{theo: fd/d decreases}, we prove that $\frac{f_{2,d}(t)}{d}$ is a decreasing function of $d$ for any given $t > 0$. Consequently, $\frac{E(2,d)}{d}$ and  $\frac{\mathit{SE}(2,d)}{d}$ are decreasing functions of $d$. 
\item In Proposition \ref{prop: condition for f increase}, we provide a condition for $f_{2,d}(t)$ being an increasing function of $d$ for any given $t > 0$. We conjecture that this condition holds true and support it by numerical simulation.
\item In Theorem \ref{theo: En2}, we provide an upper bound for $f_{n,2}(\mathbf{t})$. We also make a conjecture for $f_{n,d}(\mathbf{t})$ and $E(n,d)$ in the general case ($n,d\geq 3$).
\item We  offer numerical illustration for our main results in Section \ref{sec: simulation}.
\end{enumerate}
The density function $f_{n,d}(\mathbf{t})$ provides insightful information on the distribution of the internal equilibria: integrating $f_{n,d}(\mathbf{t})$ over any interval produces the expected number of real equilibria on that interval. In particular, the expected number of internal equilibria $E(n,d)$ is obtained by integrating $f_{n,d}(\mathbf{t})$ over the positive half of the space. Theorem \ref{theo: fd/d decreases} and Proposition \ref{prop: condition for f increase}, which are deduced from Theorems \ref{theo: fd interm of Pd} and \ref{theo: fd interms of Pd and Pd-1}, are qualitative statements, which tell us \textit{ how} the expected number of internal equilibria per unit length $f_{2,d}$ in a $d$-player two-strategy game changes when the number of players $d$ increases. Theorem \ref{theo: concentration} quantifies its behaviour showing that $f_{2,d}$ is approximately (up to a constant factor) equal to $\sqrt{d-1}$. The function $h(d):=\sqrt{d-1}$, as seen in Theorem \ref{theo: concentration}, certainly satisfies the properties that $h(d)$ increases but $\frac{h(d)}{d}$ decreases. Thus, it strengthens Theorem \ref{theo: fd/d decreases} and further supports Conjecture \ref{cojecture: fd increases}. Theorem \ref{theo: behavior of E2} is also a quantitative statement which provides an asymptotic estimate for the expected number of  internal (stable) equilibria.

Furthermore, it is important to note that the expected number of real zeros of a random polynomial has been extensively studied, dating back to 1932 with Block and P\'{o}lya's  seminal paper \cite{BP32} (see, for instance,~\cite{EK95} for a nice exposition and~\cite{TaoVu14,NNV15TMP} for the most recent progress). Therefore, our results, in Theorems \ref{theo: behavior of E2}, \ref{theo: fd interm of Pd} and \ref{theo: fd interms of Pd and Pd-1},  provide important, novel insights within the theory of random polynomials, but also reveal its intriguing  connections and applications  to EGT.

\subsection{Organisation of the paper}
The rest of the paper is structured as follows. In Section \ref{sec: pre}, we recall  relevant details on EGT and random polynomial theory. Section \ref{sec: 2d games} presents full analysis of the expected density function and the expect number of internal equilibria of a multi-player two-strategy game. The results on asymptotic behaviour and on the connection to Legendre polynomials and its applications are given in Sections \ref{subsec: asymptotic} and \ref{subsec: connection}, respectively. In Section \ref{sec: general games}, we  provide  analytical results for the two-player multi-strategy game and numerical simulations for the general case. Therein we also make a conjecture about an asymptotic formula for the density and the expected number of internal equilibria in the general case. In Section \ref{sec: conclusion}, we sum up and provide future perspectives. Finally, some detailed proofs are presented in the appendix.
\section{Preliminaries: replicator dynamics and random polynomials}
This section describes some relevant details of the EGT and random polynomial theory, to the extent we need here. Both are classical but the idea of using the latter to study the former has only been pointed out in our recent paper \cite{DH15}.
\label{sec: pre}
\subsection{Replicator dynamics}

The classical approach to evolutionary games is replicator dynamics  \cite{taylor:1978wv,zeeman:1980ze,hofbauer:1998mm,schuster:1983le,nowak:2006bo}, describing that whenever a strategy has a fitness larger than the  average fitness of the population, it is expected to  spread. Formally, let us consider an infinitely large population with $n$ strategies, numerated from 1 to $n$. They have frequencies $x_i$, $1 \leq i \leq n$, respectively, satisfying that $0 \leq x_i \leq 1$  and $\sum_{i = 1}^{n} x_i = 1$. The interaction of the individuals in the population is in randomly selected groups of $d$ participants, that is, they play and obtain their fitness from   $d$-player games.  We consider here symmetrical games (e.g. the public goods games and their generalizations \cite{hardin:1968mm,hauert:2002in,santos:2008xr,pacheco:2009aa,Han:2014tl}) in which the order of the participants is  irrelevant. Let $\alpha^{i_0}_{i_1,\ldots,i_{d-1}}$ be the payoff of the focal player, where $i_0$ ($1 \leq i_0 \leq n$) is the strategy of the focal player, and  $i_k$ (with $1 \leq i_k \leq n$ and $1 \leq k \leq d-1$) be the strategy of the player in position $k$. 
These payoffs form a $(d-1)$-dimensional payoff matrix \cite{gokhale:2010pn}, which satisfies  (because of the game symmetry)  
\begin{equation}
\label{eq:symmetry}
\alpha^{i_0}_{i_1,\ldots,i_{d-1}} = \alpha^{i_0}_{i_1',\ldots,i_{d-1}'},
\end{equation}
 whenever $\{i_1'\ldots,i_{d-1}'\}$ is a permutation of  $\{i_1\ldots,i_{d-1}\}$. This means that  only the fraction of each strategy in the game matters.

 The equilibrium points of the system are given by the points $(x_1, \dots, x_n)$ satisfying the condition that the  fitnesses of all strategies are the same, which can be simplified to the following  system of $n-1$ polynomials of degree $d-1$ \cite{DH15}
\begin{equation}
\label{eq: eqn for fitness2}
\sum\limits_{\substack{0\leq k_{1}, ..., k_{n}\leq d-1,\\ \sum^{n}\limits_{i = 1}k_i = d-1  }}\beta^i_{k_{1}, ...,k_{n-1} }\begin{pmatrix}
d-1\\
k_{1}, ..., k_n
\end{pmatrix} \prod\limits_{i=1}^{n}x_{i}^{k_i} = 0\quad \text{ for } i = 1, \dots, n-1,
\end{equation}
where $\beta^{i}_{k_{1}, ..., k_{n-1} } := \alpha^{i}_{k_{1}, ..., k_{n} } -\alpha^{n}_{k_{1}, ..., k_{n} } $, and  
$\begin{pmatrix}
d-1\\
k_{1}, \dots, k_{n}
\end{pmatrix} = \frac{(d-1)!}{\prod\limits _{k=1}^n k_i! }$
are the multinomial coefficients. Assuming that all the payoff entries have the same probability distribution, then all $\beta^{i}_{k_{1}, ..., k_{n-1} }$, $i = 1, \dots, n-1$,  have symmetric distributions, i.e. with mean 0  \cite{HTG12}. 

We focus on   internal equilibrium points \cite{gokhale:2010pn,HTG12,DH15}, i.e. $0 < x_i < 1$ for all $1 \leq i \leq n-1$. Hence, by using the transformation $t_i = \frac{x_i}{x_n}$, with $0< t_i < +\infty$ and $1 \leq i \leq n-1$, dividing the left hand side of the above equation by $x_n^{d-1}$ we obtain the following equation in terms of $(t_1,\ldots,t_{n-1})$ that is equivalent to~\eqref{eq: eqn for fitness2} 
\begin{equation}
\label{eq: eqn for fitnessy}
 \sum\limits_{\substack{0\leq k_{1}, ..., k_{n-1}\leq d-1,\\  \sum\limits^{n-1}_{i = 1}k_i \leq d-1  }}\beta^i_{k_{1}, ..., k_{n-1} }\begin{pmatrix}
d-1\\
k_{1}, ..., k_{n}
\end{pmatrix} \prod\limits_{i=1}^{n-1}t_{i}^{k_i} = 0\quad  \quad \text{ for } i = 1, \dots, n-1.
\end{equation}
Hence, finding an internal equilibria in a general  $n$-strategy $d$-player random evolutionary game  is equivalent to find a solution $(y_1, \dots, y_{n-1}) \in {\mathbb{R}_+}^{n-1}$ of the system of $(n-1)$ polynomials of degree $(d-1)$ in (\ref{eq: eqn for fitnessy}). This observation  links the study of generic properties of equilibrium points  in EGT to the theory of random polynomials.
\subsection{Random polynomial theory}

Suppose that all $\beta^{i}_{k_{1}, ..., k_{n-1} }$ are  Gaussian distributions with mean $0$ and variance $1$, then for each $i$ ($1 \leq i \leq n-1$), $A^i=\left\{\beta^i_{k_{1}, ..., k_{n-1} }\begin{pmatrix}
d-1\\
k_{1}, ..., k_{n}
\end{pmatrix}\right\}
$ is a multivariate normal random vector with mean zero and covariance matrix $C$ given by
\begin{equation}
\label{eq: matrix C}
C=\mathrm{diag} \left(\begin{pmatrix}
d-1\\
k_{1}, ..., k_{n}
\end{pmatrix}^2\right)_{0\leq k_i\leq d-1,~\sum\limits_{i=1}^{n-1}k_i\leq d-1}.
\end{equation}
The density function $f_{n,d}$ and the expected number $E(n,d)$ of equilibria can be computed explicitly. The lemma below is a direct consequence of \cite[Theorem 7.1]{EK95} (see also \cite[Lemma 1]{DH15}). For a clarity of notation, we use bold font to denote an element in high-dimensional Euclidean space such as $\mathbf{t}=(t_1,\ldots,t_{n-1})\in \mathbb{R}^{n-1}$.
\begin{lemma}
\label{lemma: E(n,d)}
Assume that $\{A^i\}_{1\leq i\leq n-1}$ are independent normal random vectors with mean zero and covariance matrix $C$ as in \eqref{eq: matrix C}. The expected density of real zeros of Eq. \eqref{eq: eqn for fitnessy} at a point $\mathbf{t}=(t_1,\ldots, t_{n-1})$ is given by
\begin{equation}
\label{eq: density 2}
f_{n,d}(\mathbf{t})=\pi^{-\frac{n}{2}}\Gamma\left(\frac{n}{2}\right)(\det L(\mathbf{t}))^\frac{1}{2},
\end{equation}
where $\Gamma$ denotes the Gamma function and $L(\mathbf{t})$ the matrix with entries
\begin{align*}
L_{ij}(\mathbf{t})&=\frac{\partial^2}{\partial x_iy_j}(\log v(\mathbf{x})^T C v(\mathbf{y}))\big|_{\mathbf{y}=\mathbf{x}=\mathbf{t}},
\end{align*}
with 
\begin{equation}
v(\mathbf{x})^T C v(\mathbf{y})=\sum\limits_{\substack{0\leq k_{1}, ..., k_{n-1}\leq d-1,\\ \sum\limits^{n-1}_{i = 1}k_i\leq d-1  }}\begin{pmatrix}
d-1\\
k_1,\ldots,k_{n}
\end{pmatrix}^2\prod\limits_{i=1}^nx_i^{k_i}y_i^{k_i}.
\end{equation}
As a consequence, the expected number of internal equilibria in a \emph{d}-player \emph{n-}strategy random game is determined by
\begin{equation}
\label{eq: expected number or EQ}
E(n,d)=\int_{\mathbb{R}_+^{n-1}}f_{n,d}(\mathbf{t})\,d\mathbf{t}.
\end{equation}
\end{lemma}
Note  that the assumption in Lemma \ref{lemma: E(n,d)} is quite limited when applying to games with more than two strategies as in that case the independence of the $\alpha$ terms does not carry over into the independence of $\beta$ terms, see Remark \ref{rem: assumption} for a detailed discussion. 

\section{Multi-player two-strategy games}
\label{sec: 2d games}
We provide mathematical analysis of the expected density function $f_{2,d}(t)$ and the expected number of equilibria $E(2,d)$ for a multi-player two-strategy game. Section \ref{subsec: asymptotic} presents asymptotic behaviour. A connection to Legendre polynomials and its applications are given in \ref{subsec: connection}. In Section \ref{subsec: connection}, further applications of this connection to study monotonicity of the density function are explored.
\subsection{Asymptotic behaviour of the density and the expected number of equilibria}
\label{subsec: asymptotic}
In the case of multi-player two-strategy games ($n=2$), the  system of polynomial equations in   \eqref{eq: eqn for fitness2} becomes a univariate polynomial equation
\begin{equation}
\label{eq: eqn for fitness d2}
\sum\limits_{k=0}^{d-1}\beta_k\begin{pmatrix}
d-1\\
k
\end{pmatrix}y^k\,(1-y)^{d-1-k}=0,
\end{equation}
where $y$ is the fraction of strategy 1 (i.e., $1-y$ is that of strategy 2) and $\beta_k$ is  the payoff to strategy 1 minus that to strategy 2 obtained in a $d$-player interaction with $k$ other participants using strategy 1.
It is worth noticing that
\begin{equation}
b_{k,d}:=\begin{pmatrix}
d-1\\
k
\end{pmatrix}y^k\,(1-y)^{d-1-k}, \quad k=0,\ldots,d-1,
\end{equation}
is the Bernstein basis polynomials of degree $d-1$ \cite{Pet99,Kow06}. Therefore, the left-hand side of \eqref{eq: eqn for fitness d2} is a random Bernstein polynomial of degree $d-1$. As a by-product of our analysis, see Theorem \ref{theo: behavior of E2}, we will later obtain an asymptotic formula of the expected real zeros of a random Bernstein polynomial.

Letting $t=\frac{y}{1-y}$ ($t  \in {\mathbb{R}_+}$), Eq. \eqref{eq: eqn for fitness d2} is simplified to
\begin{equation}
\label{eq: eqn for y}
\sum\limits_{k=0}^{d-1}\beta_k\begin{pmatrix}
d-1\\
k
\end{pmatrix}t^k=0.
\end{equation}

The expected density of real zeros of this equation at a point $t$ is $f_{2,d}(t)$.  For  simplicity of notation, from now on we write $f_{d}(t)$ instead of $f_{2,d}(t)$.  We recall some properties of the density function $f_{d}(t)$ from \cite[Proposition 1]{DH15} that will be used in the sequel.
\begin{proposition}[\cite{DH15}]
\label{pro: properties of f}
The following properties hold
\begin{enumerate}[1)]
\item $f_{d}(0)=\frac{d-1}{\pi},\quad f_{d}(1)=\frac{d-1}{2\pi}\frac{1}{\sqrt{2d-3}}$.
\item $f_{d}(t)=\frac{1}{2\pi}\left[\frac{1}{t}G'(t)\right]^\frac{1}{2}$,
where
\begin{equation}
\label{eq: Md}
G(t)=t\frac{\frac{d}{dt}M_d(t)}{M_d(t)}=t\frac{d}{dt}\log M_d(t)\quad \text{with}\quad  M_d(t)=\sum\limits_{k=0}^{d-1}\begin{pmatrix}
d-1\\
k
\end{pmatrix}^2t^{2k}.
\end{equation}
\item $f_{d}(t)$ has an alternative representation
\begin{equation}
\label{eq: fourth property}
(2\pi f_{d}(t))^2=\sum\limits_{i=1}^{d-1}\frac{4r_i}{(t^2+r_i)^2},
\end{equation}
where $r_i>0$ satisfies that $M_d(t)=\prod\limits_{i=1}^{d-1}(t^2+r_i)$.
\item $t\mapsto f_d(t)$ is a decreasing function.
\item $f_d\left(\frac{1}{t}\right)=t^2 f(t)$.
\item $E(2,d)=\int_0^\infty f_d(t)\,dt=2\int_0^1 f_d(t)\,dt$.
\end{enumerate}
\end{proposition}
\begin{example}
Below are some  examples of $f_d(t)$, with $d = 2, 3$ and 4  
\begin{equation*}
f_2(t)=\frac{1}{\pi}\frac{1}{1+t^2},\quad f_3(t)=\frac{2}{\pi}\frac{\sqrt{1+t^2+t^4}}{1+4t^2+t^4},\quad f_4(t)=\frac{3}{\pi}\frac{\sqrt{1+4t^2+10t^4+4t^6+t^8}}{1+9t^2+9t^4+t^6}.
\end{equation*}
\end{example} 

We recall that $t=\frac{y}{1-y}$, where $y$ is the fraction of strategy 1. We can write the density in terms of $y$ using the change of variable formula as follows.
\begin{equation*}
f_d(t)\,dt=f_d\Big(\frac{y}{1-y}\Big)\frac{1}{(1-y)^2}\,dy.
\end{equation*}
Define $g_d(y)$ to be the density on the right-hand side of the above expression, i.e.,
\begin{equation}
g_d(y):=f_d\Big(\frac{y}{1-y}\Big)\frac{1}{(1-y)^2}.
\end{equation}
The following lemma is an interesting property of the density function $g_d$, which explains the symmetry of the strategies (swapping the index labels converts an equilibrium at $y$ to one at $1-y$).
Numerical simulations in Section \ref{sec: simulation} further illustrate this property (see Figure \ref{fig:f2d_depend_frequency_x}).
\begin{lemma} 
\label{lemma:symmetry of g}
The function $y\mapsto g_d(y)$ is symmetric about the line $y=\frac{1}{2}$, i.e.,
\begin{equation*}
g_d(y)=g_d(1-y).
\end{equation*}
\end{lemma}
\begin{proof}
We have
\begin{align*}
g_d(1-y)=f_d\Big(\frac{1-y}{y}\Big)\frac{1}{y^2}=f_d\Big(\frac{y}{1-y}\Big)\frac{y^2}{(1-y)^2}\frac{1}{y^2}=f_d\Big(\frac{y}{1-y}\Big)\frac{1}{(1-y)^2}=g_d(y),
\end{align*}
where we have used the  fifth property in Proposition \ref{pro: properties of f} to obtain the second equality above.
\end{proof}

The following  theorem provides an upper bound and asymptotic behaviour for $f_{d}(t)$ and its scaling with respect to $d$.
\begin{theorem}[Asymptotic behaviour of the density function] The following statement holds
\label{theo: concentration}
\begin{equation}
\label{eq: estimate for f}
f_{d}(t)\leq \min\left\{\frac{\sqrt{d-1}}{2\pi\,t},\frac{d-1}{\pi}\right\} \quad \text{for}\quad t\geq 0.
\end{equation}
As a consequence, for any given $t>0$
\begin{equation}
\label{eq: limit of f}
\frac{f_{d}(t)}{d-1}\leq \frac{1}{\pi},\quad\lim_{d\rightarrow \infty}\frac{f_{d}(t)}{d-1}=0,
\end{equation}
and furthermore
\begin{equation}
f_{d}(t)\sim \sqrt{d-1}.
\end{equation}
\end{theorem}
\begin{proof}
It follows from the first and the fourth properties in Proposition \ref{pro: properties of f} that
\begin{equation}
\label{eq: first bound of f}
f_{d}(t)\leq f_{d}(0)=\frac{d-1}{\pi}.
\end{equation}
On the other hand, according to the third property in Proposition \ref{pro: properties of f}, we have
\begin{equation*}
(2\pi f_{d}(t))^2=\sum_{i=1}^{d-1}\frac{4 r_i}{(t^2+r_i)^2}.
\end{equation*}
Since $(t^2+r_i)^2\geq 4\, r_it^2$, it follows that for any $t\neq 0$
\begin{equation*}
(2\pi f_{d}(t))^2\leq \sum_{i=1}^{d-1}\frac{1}{t^2}=\frac{d-1}{t^2},
\end{equation*}
which is equivalent to $ f_{d}(t)\leq \frac{\sqrt{d-1}}{2\pi\,t}$. Hence, the upper bound  of $f_{d}(t)$ in \eqref{eq: estimate for f} holds. 

The upper bound and limit in \eqref{eq: limit of f} are obvious consequences of \eqref{eq: estimate for f}. It remains to prove the asymptotic behaviour of $f_{d}(t)$ as a function of $d$. Using \eqref{eq: estimate for f} and the fourth property in Proposition \ref{pro: properties of f}, it follows that, for $0< t\leq 1$
\begin{equation*}
\frac{d-1}{2\pi}\frac{1}{\sqrt{2d-3}}= f_d(1)\leq f_{d}(t)\leq \frac{\sqrt{d-1}}{2\pi\,t}, 
\end{equation*}
and for $t\geq 1$
\begin{equation*}
\frac{1}{t^2}\frac{d-1}{2\pi}\frac{1}{\sqrt{2d-3}}=\frac{1}{t^2}f_d(1)\leq \frac{1}{t^2}f_{d}\left(\frac{1}{t}\right)=f_{d}(t)\leq f_d(1)=\frac{d-1}{2\pi}\frac{1}{\sqrt{2d-3}}.
\end{equation*}
From these estimates, we deduce that $f_{d}(t)\sim \sqrt{d-1}$ for any $t>0$.
\end{proof}
We numerically illustrate Theorem \ref{theo: concentration} in Section \ref{sec: simulation}, see Figures \ref{fig:fd}.

\begin{theorem}[Asymptotic behaviour of $E(2,d)$]
\label{theo: behavior of E2}
It holds that 
\begin{equation}
\label{eq: improve upper bound}
\sqrt{d-1}\lesssim E(2,d)\lesssim \sqrt{d-1}\ln(d-1).
\end{equation}
Furthermore, we obtain the following asymptotic formula for $E(2,d)$
\begin{equation}
\label{eq: asymptotic of E}
\lim\limits_{d\rightarrow\infty}\frac{\ln E(2,d)}{\ln(d-1)}=\frac{1}{2}.
\end{equation}
\end{theorem}
We first provide a proof of this theorem. Then we discuss its implications for the expected number of stable equilibrium points in a random game and that of real zeros of a random Bernstein polynomial in Corollaries \ref{cor: expected stable equi} and \ref{cor: expected zeros Bernstein}. A comparision with well-known results in the theory of random polynomials is given in Remark \ref{re: elliptic polynomial}.
\begin{proof}
The lower bound was derived previously in \cite[Theorem 2]{DH15}. For the sake of completeness, we provide it again here. Using the sixth and the fourth properties in Proposition \ref{pro: properties of f}, we have $E(2,d)=2\int_0^1 f_{d}(t)\,dt\geq 2 \int_0^1 f(1)\,dt=2 f(1)=\frac{d-1}{\pi\sqrt{2d-3}}$. Therefore, $\sqrt{d-1}\lesssim E(2,d)$.

The underlying idea of the proof for the upper bound is to split  the integral range  in the formula of $E(2,d)$ into two intervals. The first one is from $0$  to $\alpha$, for some  $\alpha \in (0,1)$; we then  estimate $f_{d}(t)$ in this interval by $f_{d}(0)$. The second  one is from $\alpha$ to $1$, which is estimated  using the upper bound of $f_{d}(t)$ given in \eqref{eq: estimate for f}. The value of $\alpha$ will then be optimized.
\begin{align}
E(2,d)&=2\int_0^1 f_{d}(t)\,dt=2\left[\int_0^\alpha f_{d}(t)\,dt+\int_\alpha^1 f_{d}(t)\,dt\right]\nonumber
\\&\leq 2\left[\alpha f_{d}(0)+\frac{\sqrt{d-1}}{2\pi}\int_\alpha^1\frac{1}{t} \,dt\right]\nonumber
\\&=\frac{1}{\pi}\left[2(d-1)\,\alpha-\sqrt{d-1}\,\ln \alpha\right].\label{E2: 1}
\end{align}
Let $h(\alpha)$ be the expression inside the square brackets in the right-hand side of \eqref{E2: 1}. To obtain the optimal (i.e. smallest) upper bound, we minimize $h(\alpha)$ with respect to $\alpha$. The optimal value of $\alpha$ satisfies the following equation
\begin{equation*}
\frac{d}{d\alpha}h(\alpha)=2(d-1)-\frac{\sqrt{d-1}}{\alpha}=0,
\end{equation*}
which leads to $\alpha=\frac{1}{2\sqrt{d-1}}$. Substituting this value into \eqref{E2: 1}, we obtain
\begin{equation*}
E(2,d)\leq \frac{\sqrt{d-1}}{\pi}\left(1+\ln 2+\frac{1}{2}\ln (d-1)\right).
\end{equation*}
It follows that $E(2,d)\lesssim \sqrt{d-1}\ln(d-1)$, which is \eqref{eq: improve upper bound}.

We now prove \eqref{eq: asymptotic of E}. By taking logarithm in  \eqref{eq: improve upper bound}, we obtain
\begin{equation*}
\frac{1}{2}\ln(d-1)\lesssim \ln E{(2,d)}\lesssim \frac{1}{2}\ln(d-1)+\ln\ln(d-1).
\end{equation*}
Therefore
\begin{equation}
\frac{1}{2}\lesssim \frac{\ln E(2,d)}{\ln(d-1)}\lesssim \frac{1}{2}+\frac{\ln\ln(d-1)}{\ln(d-1)}.
\end{equation}
Since $\lim\limits_{d\rightarrow\infty}\frac{\ln\ln(d-1)}{\ln(d-1)}=0$, we achieve \eqref{eq: asymptotic of E}. 
\end{proof}

Theorem \ref{theo: behavior of E2} has two interesting implications about the expected number of stable equilibrium points in a random game and that of real zeros of a random Bernstein polynomial.
\begin{corollary}
\label{cor: expected stable equi}
The expected number of stable equilibrium points in a random game with $d$ players and two strategies, $\mathit{SE}(2,d)$, satisfies 
\begin{equation}
\label{eq: improve upper bound SE}
\frac{\sqrt{d-1}}{2}\lesssim \mathit{SE}(2,d)\lesssim \frac{1}{2}\sqrt{d-1}\ln(d-1), 
\end{equation}
and furthermore, satisfies  the following limiting  behaviour  
\begin{equation}
\label{eq: asymptotic of SE}
\lim\limits_{d\rightarrow\infty}\frac{\ln \mathit{SE}(2,d)}{\ln(d-1)}=\frac{1}{2}.
\end{equation}
\end{corollary}
\begin{proof} 
From \cite[Theorem 3]{HTG12}, it is known that  an equilibrium in a random game with two strategies and an arbitrary number of players, is stable with probability $1/2$. Thus, $\mathit{SE}(2,d) = \frac{E(2,d)}{2}$. Hence, the corollary is clearly followed from  Theorem  \ref{theo: behavior of E2}. 
\end{proof}
\begin{corollary}
\label{cor: expected zeros Bernstein}
The expected number of  real zeros, $E_{\B}$, of a random Bernstein polynomial
\begin{equation*}
\B(x)=\sum\limits_{k=0}^{d-1}\beta_k\begin{pmatrix}
d-1\\
k
\end{pmatrix}x^k\,(1-x)^{d-1-k},
\end{equation*} 
where $\beta_k$ are independent standard normal distributions, satisfies
\begin{equation}
\label{eq: EB estimate}
\sqrt{d-1}\lesssim E_{\B}\lesssim \sqrt{d-1}\ln(d-1),\quad \lim_{d\rightarrow\infty}\frac{\ln E_{\B}}{\ln (d-1)}=\frac{1}{2}.
\end{equation}
\end{corollary}
\begin{proof}
As mentioned beneath \eqref{eq: eqn for fitness d2}, solving $\B(x)=0$ is equivalent to solving Eq. \eqref{eq: eqn for y}. It follows that $f_d(t)$ is the expected density of real zeros of $\B(x)$. Therefore, $E_{\B}$ given by
\begin{equation}
\label{eq: EB}
E_{\B}=\int_{-\infty}^\infty f_d(t)\,dt=2\int_0^\infty f_d(t)\, dt=2 E(2,d).
\end{equation}
Note that the second equality in \eqref{eq: EB} holds because $f_d(t)$ is even in $t$ due to \eqref{eq: fourth property}. The asymptotic behaviour \eqref{eq: EB estimate} of $E_\B$ is then followed directly from Theorem \ref{theo: behavior of E2}.
\end{proof}
\begin{remark}
\label{re: elliptic polynomial}
The study of the distribution and expected number of real zeros of a random polynomial is an active research field with a long history dating back to 1932 with Block and P\'{o}lya \cite{BP32}, see for instance~\cite{EK95} for a nice exposition and~\cite{TaoVu14,NNV15TMP} for the most recent results and discussions. Consider a random polynomial of the type
\begin{equation}
\P_d(z)=\sum_{i=0}^{d-1} c_i\,\xi_i\,z^i.
\end{equation}
The most well-known classes of polynomials studied extensively in the literature are: flat polynomials or Weyl polynomials for $c_i\colonequals\frac{1}{i!}$, elliptic polynomials (EP) or binomial polynomials for $c_i\colonequals\sqrt{\begin{pmatrix}
d-1\\
i
\end{pmatrix}}$ and Kac polynomials for $c_i\colonequals 1$. We emphasize the difference between the polynomial studied in this paper, i.e. the right-hand side of Eq. \eqref{eq: eqn for y}, with the elliptic polynomial: in our case $c_i=\begin{pmatrix}
d-1\\
i
\end{pmatrix}$ are binomial coefficients, not their square root as in the elliptic polynomial. In the former case, $v(x)^TCv(y)=\sum_{i=1}^{d-1}\begin{pmatrix}
d-1\\
i
\end{pmatrix}x^iy^i=(1+xy)^{d-1}$, and as a result the density function and the expected number of real zeros have closed formula, see \cite{EK95}
\begin{equation*}
f_{EP}(t)=\frac{\sqrt{d-1}}{\pi (1+t^2)}, \quad E_{EP}=\sqrt{d-1}.
 \end{equation*} 
Our case is more challenging, because of the square in the coefficients,\\
$v(x)^TCv(y)=\sum_{i=1}^{d-1}\begin{pmatrix}
d-1\\
i
\end{pmatrix}^2x^iy^i$ is no longer a generating function. Nevertheless, Theorem \ref{theo: behavior of E2} shows that $E(2,d)$ still has interesting asymptotic behaviour as in \eqref{eq: improve upper bound} and \eqref{eq: asymptotic of E}.
\end{remark}
\subsection{Connections to Legendre polynomials and other qualitative properties}
\label{subsec: connection}
In this section, we first establish a connection between the expected density function $f_d$ and the well-known Legendre polynomials. Then using the connection and  known properties of Legendre polynomials, we prove some qualitative properties of $f_d$ and the expected number of equilibria. The main results of this section can be summarised as follows.
\begin{enumerate}[(i)]
\item  Theorem \ref{theo: fd interm of Pd} and Theorem \ref{theo: fd interms of Pd and Pd-1} derive an expression for  the expected density $f_d$ in terms of the Legendre polynomials.
\item  Theorem \ref{theo: fd/d decreases} shows that $\frac{f_d(t)}{d-1}$ is an increasing function of $d$ for any given $t>0$.
\item Corollary \ref{cor: monotonicity of E} proves that $\frac{E(2,d)}{d-1}$ and $\frac{SE(2,d)}{d-1}$ are decreasing functions of $d$.
\end{enumerate}
Technically,  keys to these theorems are Lemma \ref{lem: connection btw Md and Pd} that connects the Legendre polynomials $P_d$ to $M_{d+1}$ in \eqref{eq: Md}, and Lemma \ref{lem: Turan inequality}  showing a reverse Turan's inequality. These lemmas are of interest in their own right.

\subsubsection{Legendre polynomials}
We recall some relevant details  on Legendre polynomials. Legendre polynomials, denoted by $P_d(x)$, are solutions to Legendre's differential equation
\begin{equation}
\label{eq: Legendre equation}
\frac{d}{dx}\left[(1-x^2)\frac{d}{dx}P_d(x)\right]+d(d+1)P_d(x)=0,
\end{equation}
with initial data $P_0(x)=1, \ \ P_1(x)=x$. The following are some important properties of the Legendre polynomials that will be used in the sequel.
\begin{enumerate}[(1)]
\item Explicit representation 
\begin{equation}
\label{eq: explicit reprentation of Legendre}
P_d(x)=\frac{1}{2^d}\sum_{i=0}^{d}\begin{pmatrix}
d\\
i
\end{pmatrix}^2(x-1)^{d-i}(x+1)^i.
\end{equation}
\item Recursive relation
\begin{equation}
\label{eq: recursive of legendre}
(d+1)P_{d+1}(x)=(2d+1)xP_d(x)-dP_{d-1}(x).
\end{equation}
\item First derivative relation
\begin{equation}
\label{eq: first derivative  Legendre}
\frac{x^2-1}{d}P_d'(x)=xP_d(x)-P_{d-1}(x).
\end{equation}
\end{enumerate}

\begin{example} The first few Legendre polynomials are
\begin{align*}
&P_0(x)=1,\quad P_1(x)=x,\quad P_2(x)=\frac{1}{2}(3x^2-1),\quad P_3(x)=\frac{1}{2}(5x^3-3x),
\\&\quad P_4(x)=\frac{1}{8}(35x^4-30x^2+3).
\end{align*}
\end{example}

The Legendre polynomials were first introduced in 1782 by A. M. Legendre as the coefficients in the expansion of the Newtonian potential \cite{Legendre1785}. This expansion gives the gravitational potential associated to a point mass or the Coulomb potential associated to a point charge. In the course of time, Legendre polynomials have been widely used in Physics and Engineering. For instance, they occur when one solves Laplace's equation (and related partial differential equations) in spherical coordinates \cite{Bayin06}.

Our approach reveals an appealing, and previously unexplored relationship between Legendre polynomials and evolutionary game theory. We explore this relationship and its applications in the next sections.

\subsubsection{From Legendre polynomials to evolutionary games}
The starting point of our analysis is the following relation between the polynomial $M_{d+1}(t)$ in \eqref{eq: Md} and the Legendre polynomial $P_d$. 
\begin{lemma}
\label{lem: connection btw Md and Pd}
It holds that
\begin{equation}
\label{eq: connection btw Md and Pd}
M_{d+1}(t)=(1-t^2)^d\,P_d\left(\frac{1+t^2}{1-t^2}\right).
\end{equation}
\end{lemma}
\begin{proof}
See Appendix \ref{app: lem: connection btw Md and Pd}.
\end{proof}

Note that traditionally in the Legendre polynomials, the arguments are in $[-1,1]$. In this paper, however, the arguments are not in this interval since $\big|\frac{1+t^2}{1-t^2}\big|\geq 1$. Legendre polynomials with arguments greater than unity have been used in the literature, for instance in \cite[Chapter 2]{Cur03}.
We now establish a connection between the density $f_d(t)$ and the Legendre polynomials. According to the second property in Proposition \ref{pro: properties of f}, we have
\begin{equation}
\label{eq: fd interms of Md}
f_{d+1}(t)=\frac{1}{2\pi}\left[\frac{1}{t}\left(t\frac{M_{d+1}'(t)}{M_{d+1}(t)}\right)'\right]^\frac{1}{2},
\end{equation}
where $'$ denotes the derivative with respect to $t$. Using this formula and Lemma \ref{lem: connection btw Md and Pd}, we obtain the following expression of $f_{d+1}(t)$ in terms of $P_d(t)$ and its derivative. 
 
\begin{theorem}[Expression of the density in terms of the Legrendre polynomial and its derivative]
\label{theo: fd interm of Pd}
The following formula holds
\begin{equation}
\label{eq: fd interms of Pd}
(2\pi\,f_{d+1}(t))^2=\frac{4d^2}{(1-t^2)^2}-\frac{16t^2}{(1-t^2)^4}\left(\frac{P'_d}{P_d}\right)^2\left(\frac{1+t^2}{1-t^2}\right)
\end{equation}
\end{theorem}
\begin{proof}
See Appendix \ref{app: proof of Theorem fd interm of Pd}.
\end{proof}

\begin{corollary}
As a direct consequence of \eqref{eq: fd interms of Pd}, we obtain the following bound for $f_{2,d}(t)$.
\begin{equation}
\label{eq: improved upper bound}
f_d(t)\leq \frac{1}{\pi}\frac{d-1}{1-t^2}.
\end{equation}
In comparison with the estimate \eqref{eq: estimate for f} obtained in Theorem \ref{theo: concentration}, this inequality is weaker for $t>0$ since it is of order $O(d-1)$. However, it does not blow up as $t$ approaches  $0$.
\end{corollary}
We provide another expression of $f_{d+1}(t)$ in terms of two consecutive Legendre polynomials $P_{d-1}$ and $P_{d}$. In comparison with \eqref{eq: fd interms of Pd}, this formula avoids the computations of the derivative of the Legendre polynomial $P_d$. 
\begin{theorem}[Expression of the density function in terms of two Legendre polynomials]
\label{theo: fd interms of Pd and Pd-1}
It holds that
\begin{equation}
\label{eq: main fd interms of Pd}
(2\pi f_{d+1}(t))^2=\frac{4d^2}{(1-t^2)^2}-\frac{d^2}{t^2}\left[\frac{1+t^2}{1-t^2}-\frac{P_{d-1}}{P_d}\left(\frac{1+t^2}{1-t^2}\right)\right]^2.
\end{equation}
\end{theorem}
\begin{proof}
See Appendix \ref{app: proof of Theorem fd interms of Pd and Pd-1}
\end{proof}
\subsubsection{Monotonicity of the densities}
\label{subsec: monotone}
Theorem \ref{theo: fd interms of Pd and Pd-1} is crucial for  the subsequent qualitative study of the density $f_{2,d}(t)$ for varying $d$. 

\begin{lemma}
\label{lem: Turan inequality}
The following inequality holds for all $|x|\geq 1$ 
\begin{equation}
P_d(x)^2\leq P_{d+1}(x)P_{d-1}(x).
\end{equation}
\end{lemma}
\begin{proof}
See Appendix \ref{app: proof of lemma Turan inequality}.
\end{proof}

Note that this inequality is the reverse of the Tur\'{a}n inequality \cite{Turan50} where the author considered the case $x\in [-1,1]$.


\begin{theorem}
\label{theo: fd/d decreases}
For any given $t  > 0$, $\frac{f_{d}(t)}{d-1}$ is a decreasing  function of $d$.
\end{theorem}
\begin{proof}
We need to prove that
\begin{equation*}
\frac{f_{d+1}(t)}{d}\geq \frac{f_{d+2}(t)}{d+1}.
\end{equation*}
From \eqref{eq: main fd interms of Pd}, we have
\begin{equation}
\label{eq: fd/d interms of Pd}
4\pi^2 \left(\frac{f_{d+1}(t)}{d}\right)^2=\frac{4}{(1-t^2)^2}-\frac{1}{t^2}\left[\frac{1+t^2}{1-t^2}-\frac{P_{d-1}}{P_d}\left(\frac{1+t^2}{1-t^2}\right)\right]^2.
\end{equation}
Assume that $x\geq 1$. Then $P_d(x)>0$ for all $d$ and from \eqref{eq: recursive of legendre}, we have
\begin{equation*}
(d+1)P_{d+1}(x)\geq (2d+1)\,P_d(x)-d\,P_{d-1}(x),
\end{equation*}
which implies that
\begin{equation*}
(d+1)(P_{d+1}(x)-P_{d}(x))\geq d(P_d(x)-P_{d-1}(x))\geq \cdots\geq 1(P_1(x)-P_0(x))=x-1\geq 0.
\end{equation*}
Therefore $P_{d+1}(x)\geq P_d(x)$ for all $x\geq 1$.
We first consider the case $0\leq t<1$. From Lemma \ref{lem: Turan inequality} for $x=\frac{1+t^2}{1-t^2}\geq 1$, we have
\begin{equation*}
P_d^2\left(\frac{1+t^2}{1-t^2}\right)\leq P_{d-1}\left(\frac{1+t^2}{1-t^2}\right)P_{d+1}\left(\frac{1+t^2}{1-t^2}\right),
\end{equation*}
or equivalently
\begin{equation*}
\frac{P_{d}}{P_{d+1}}\left(\frac{1+t^2}{1-t^2}\right)\leq \frac{P_{d-1}}{P_d}\left(\frac{1+t^2}{1-t^2}\right)\leq 1\leq \frac{1+t^2}{1-t^2}.
\end{equation*}
It follows that
\begin{equation*}
\left[\frac{1+t^2}{1-t^2}-\frac{P_{d-1}}{P_d}\left(\frac{1+t^2}{1-t^2}\right)\right]^2\leq \left[\frac{1+t^2}{1-t^2}-\frac{P_{d}}{P_{d+1}}\left(\frac{1+t^2}{1-t^2}\right)\right]^2.
\end{equation*}
From this inequality and \eqref{eq: fd/d interms of Pd} as well as a similar formula for $f_{d+2}$, we obtain
\begin{equation*}
\frac{f_{d+1}(t)}{d}\geq  \frac{f_{d+2}(t)}{d+1},
\end{equation*}
which is the claimed property for the case $0\leq t<1$. The case $t\geq 1$ is then followed due to the relation $f_{d}(t)=\frac{1}{t^2}f_d\left(\frac{1}{t}\right)$.
\end{proof}
Since $E(2,d)=\int_0^\infty f_{d}(t)\,dt$ and $\mathit{SE}(2,d)=\frac{E(2,d)}{2}$, the following statement is an obvious consequence of Theorem \ref{theo: fd/d decreases}.
\begin{corollary}
\label{cor: monotonicity of E}
$\frac{E(2,d)}{d-1}$ and $\frac{\mathit{SE}(2,d)}{d-1}$ are decreasing functions of $d$.
\end{corollary}
The following proposition is a necessary and sufficient condition, which is stated in terms of the Legendre polynomials, for $f_{d+1}(t)$ being  increasing as a function of $d$.
\begin{proposition}
\label{prop: condition for f increase}
$f_{d+1}(t)$ is an increasing  function of $d$ if and only if, for $x=\frac{1+t^2}{1-t^2}$
\begin{align}
\label{eq: condition for fd increase}
&(d+1)^2[P^2_{d+1}(x)-P^2_{d}(x)]\cdot [P^2_{d+1}(x)+P^2_{d}(x)-2x P_{d+1}(x)P_d(x)]\nonumber
\\&\qquad+(2d+1)(x^2-1)P^2_d(x)P^2_{d+1}(x)\geq 0.
\end{align}
Furthermore, if 
\begin{equation} 
\label{eq: condition for fd increase 2}
(2d+1)P_d^4 \geq   P_{d-1}^2 \left[ (2d-1) P_{d+1}^2 + 2 P_d^2  \right] 
\end{equation}
then \eqref{eq: condition for fd increase} holds.
\end{proposition}
\begin{proof}
See Appendix \ref{app: proof of lemma condition for f increase}.
\end{proof}
We numerically verify the inequality \eqref{eq: condition for fd increase 2} in Figure \ref{fig:fd increase}; however, it is unclear to us how to prove it rigorously. We also recall that it is shown in Theorem \ref{theo: concentration} that $f_d(t)$ behaves like $\sqrt{d-1}$, which is an increasing function of $d$, as $d$ is sufficiently large. Motivated by these observations, we make the following prediction.
\begin{conjecture}
\label{cojecture: fd increases}
For any given $t > 0$, $f_d(t)$ is an  increasing function of $d$.
\end{conjecture}


We provide further numerical simulation to support this Conjecture by directly plotting $f_d(t)$ in Figure \ref{fig:fd}c.
\begin{remark}
\label{remark:from_t_to_x}
We recall that in the case $n=2$, the variable $t$ is defined by $t=\frac{y}{1-y}$, where $y$ is the fraction of strategy $1$ and $1-y$ is that of strategy $2$. Equivalently, $y$ can be expressed in terms of $t$ as $y=\frac{t}{1+t}$. Hence one can also transform the statements of the theorem above (and later) in terms of $y$. As has been shown in the beginning of Section \ref{sec: pre}, the transformation from $y$ to $t$ has the advantage that it transforms a complex equation (Eq. \eqref{eq: eqn for fitness d2}) to a univariate polynomial equation (Eq. \eqref{eq: eqn for y}). This enables us to exploit many available results and techniques from the literature of random polynomial theory.
Moreover, from the relationship between $f(t)$ and $g(y)$, it follows that all the monotonicity properties with respect to $d$ are reserved for $g_{2,d}$  (see a numerical illustration in Figure \ref{fig:f2d_depend_frequency_x}).
\end{remark}

\section{General Cases}
\label{sec: general games}
In this section, first we prove an estimate for the density $f_{n,2}(\mathbf{t})$ similarly as in Theorem \ref{theo: concentration} for a two-player multi-strategy game. The expected number of internal equilibria in this case has been computed explicitly in \cite[Theorem 3]{DH15}. We then conjecture an asymptotic formula for the general case. Finally, we provide numerical simulations to support our conjectures, as well as  the main results in the previous sections.
\subsection{Two-player multi-strategy games}
\begin{theorem}[two-player multi-strategy games]\\
\label{theo: En2}
Assume that $\{A^i=(\beta_j^i)_{j=1,\cdots,n-1}; i=1,\ldots, n-1\}$ are independent random vectors, then it holds that
\begin{align}
\label{eq: estimate of f2n}
& f_{n,2}(\mathbf{t})=\pi^{-\frac{n}{2}}\Gamma\left(\frac{n}{2}\right) \quad \text{if}~~ t_1\cdots t_{n-1}=0 \quad \text{and}
\\& f_{n,2}(\mathbf{t})\leq \pi^{-\frac{n}{2}}\Gamma\left(\frac{n}{2}\right)\min\left\{1,\frac{1}{n^{\frac{n}{2}}\,t_1\,\cdots\, t_{n-1}}\right\}\quad \text{if}~~ t_1\cdots t_{n-1}\neq 0. \label{eq: estimate of f2n2}
\end{align}
As a consequence, for any $\mathbf{t}$ such that $t_1\cdots t_{n-1}\neq 0$
\begin{equation}
\lim_{n\rightarrow\infty}\frac{f_{n,2}(\mathbf{t})}{\pi^{-\frac{n}{2}}\Gamma\left(\frac{n}{2}\right)}=0.
\end{equation}
\end{theorem}
\begin{proof}
According to \cite[Theorem 3]{DH15}
\begin{equation}
f_{n,2}(\mathbf{t})=\pi^{-\frac{n}{2}}\Gamma\left(\frac{n}{2}\right)\frac{1}{\left(1+\sum\limits_{k=1}^{n-1} t_k^2\right)^\frac{n}{2}}.
\end{equation}
The first assertion is then followed. Now suppose that $t_1\cdots t_{n-1}\neq 0$. By the Cauchy-Schwartz inequality, which states $\sum_{i=1}^n a_n\geq n (a_1\cdots a_n)^\frac{1}{n}$ for $n$ all $n$ positive numbers $a_1,\cdots,a_n$, we have
\begin{equation*}
\left(1+\sum\limits_{k=1}^{n-1} t_k^2\right)^\frac{n}{2}\geq n^{\frac{n}{2}}\,t_1\,\cdots\, t_{n-1}.
\end{equation*}
Therefore
\begin{equation*}
f_{n,2}(\mathbf{t})\leq \pi^{-\frac{n}{2}}\Gamma\left(\frac{n}{2}\right)\frac{1}{n^{\frac{n}{2}}\,t_1\,\cdots\, t_{n-1}}.
\end{equation*}
On the other hand, since $\left(1+\sum\limits_{k=1}^{n-1} t_k^2\right)^\frac{n}{2}\geq 1$, it follows that
\begin{equation*}
f_{n,2}(\mathbf{t})\leq  \pi^{-\frac{n}{2}}\Gamma\left(\frac{n}{2}\right).
\end{equation*}
Therefore, we obtain \eqref{eq: estimate of f2n2}.
\end{proof} 
Note that the expected number of internal equilibria for a two-player multi-strategy game is given by \cite[Theorem 3]{DH15}
\begin{equation*}
E(n,2)=\frac{1}{2^{n-1}}.
\end{equation*}
\begin{table}[ht]
\caption{Expected number of internal equilibria for  $n = 3$ and different $d$, as generated from $E(3,d)$ and from averaging over $10^6$ random samples  of the system of equation in \ref{eq: eqn for fitness2}. In all cases, results from random samplings are slightly smaller $E(3,d)$.}  
\vspace{0.2cm}
\begin{tabular}{l c l*{7}{c}r}
d                     & 2 & 3 & 4 & 5  \\
\hline
\hline
$E(3,d)$ from Theory           & 0.25 & 0.57 & 0.92 & 1.29 \\
$E(3,d)$  from Random Sampling            & 0.249496 & 0.569169 & 0.910236 & 1.28898   \\
\end{tabular}

\label{table:Ed}
 \end{table}

\begin{remark}
\label{rem: assumption}
In Theorem \ref{theo: En2}, we need an assumption that
the random vectors $\{A^i=(\beta_j^i)_{j=1,\cdots,n-1}; i=1,\ldots, n-1\}$ are independent. Recalling from Section \ref{sec: pre} that $\beta^{i}_{k_{1}, ..., k_{n-1} }= \alpha^{i}_{k_{1}, ..., k_{n} } -\alpha^{n}_{k_{1}, ..., k_{n} } $, where  $\alpha^{i}_{i_1,\ldots,i_{d-1}}$ is the payoff of the focal player, with  $i$ ($1 \leq i \leq n$) being the strategy of the focal player, and $i_k$ (with $1 \leq i_k \leq n$ and $1 \leq k \leq d-1$) is the strategy of the player in position $k$. The assumption clearly holds for $n = 2$. For $n>2$, the assumption  holds only under quite restrictive conditions such as $\alpha^{n}_{k_{1}, ..., k_{n} }$ is deterministic or $\alpha^{i}_{k_{1}, ..., k_{n} }$ are essentially identical. Note that the assumption is necessary to apply \cite[Theorem 7.1]{EK95}.  Hence, to remove this assumption, one would need to generalize \cite[Theorem 7.1]{EK95}. This is difficult and is still an open problem~\cite{Kos93,Kos01}. Nevertheless, since the system~\eqref{eq: eqn for fitnessy} has not been studied in the mathematical literature, it is interesting on its own to investigate the number of real zeros of this system under the assumption of independence of $\{A^i\}$. As such, the investigation not only provides new insights into the theory of zeros of systems of random polynomials but also gives important hint on the complexity of the game theoretical question, i.e. the number of expected number of equilibria. In Figure \ref{fig:theory vs samplings} and Table \ref{table:Ed}, we numerically compare the density function $g_{3,d}$ and the expected number of equilibria $E_{3,d}$ computed from the theory with the above assumption and from samplings without the assumption. We observe that $g_{3,d}$ have the same shape (behaviour) in both cases. In addition, the values of $E_{3,d}$ computed from samplings are slightly smaller than  those  computed from the theory.
\end{remark}

We now make a conjecture for the expected density and expected number of equilibria in a general multi-player multi-strategy game.
\begin{conjecture} 
\label{con: conjecture}
In a multi-player multi-strategy game, it holds that 
\begin{equation}
f_{n,d}\sim (d-1)^\frac{n-1}{2}\quad\text{and}\quad \lim_{d\rightarrow\infty}\frac{\ln E(n,d)}{\ln (d-1)}=\frac{n-1}{2}.
\end{equation}
We envisage that even a stronger statement holds
\begin{equation}
E(n,d)\sim (d-1)^\frac{n-1}{2}.
\end{equation}
\end{conjecture}
This conjecture is motivated from Theorem \ref{theo: behavior of E2} and a similar result for the system of multivariate elliptic polynomial as follows (see also Remark \ref{re: elliptic polynomial}). Consider a system of $n-1$ random polynomials, each of the form
\begin{equation}
\label{eq: eqn for fitness3}
\sum\limits_{\substack{0\leq k_{1}, ..., k_{n}\leq d-1,\\ \sum^{n}\limits_{i = 1}k_i = d-1  }}\beta_{k_{1}, ...,k_{n} }\sqrt{\begin{pmatrix}
d-1\\
k_{1}, ..., k_n
\end{pmatrix}} \prod\limits_{i=1}^{n}x_{i}^{k_i},
\end{equation}
where the coefficients $\beta_{k_{1}, ...,k_{n} }$ are independent standard multivariate normal distribution. Then according to \cite[Section 7.2]{EK95}, the expected density and the expected number of real zeros are given by
\begin{equation}
\label{eq: EP system}
f_{EP}(\mathbf{t})=\pi^{-\frac{n}{2}}\Gamma\left(\frac{n}{2}\right)\frac{(d-1)^{\frac{n-1}{2}}}{(1+\mathbf{t}\cdot \mathbf{t})^\frac{n}{2}},\quad E_{EP}=(d-1)^\frac{n-1}{2}.
\end{equation}
These formula are direct consequences of Lemma \ref{lemma: E(n,d)} and the fact that $v(\mathbf{x})C v(\mathbf{y})$ is a generating function, which generalises the univariate case 
\begin{equation*}
v(\mathbf{x})^TC v(\mathbf{y})=\sum\limits_{\substack{0\leq k_{1}, ..., k_{n}\leq d-1,\\ \sum^{n}\limits_{i = 1}k_i = d-1  }}\begin{pmatrix}
d-1\\
k_{1}, ..., k_n
\end{pmatrix} \prod\limits_{i=1}^{n}x^{k_i}y^{k_i} = (1+\mathbf{x}\cdot\mathbf{y})^{d-1}.
\end{equation*}
As mentioned in Remark \ref{re: elliptic polynomial}, in our case $v(\mathbf{x})^TC v(\mathbf{y})$ is no longer a generating function due to the square of the multinomial coefficients. Motivated by this and Theorem \ref{theo: behavior of E2} for the univariate case, we expect that the conjecture holds for the general case (i.e. $d$-player $n$-strategy normal evolutionary games). 
\subsection{Numerical results}
\label{sec: simulation}
In this section, we provide numerical simulations for all the (main) results obtained in the previous sections. The following figures are plotted. In the following list, (1) to (4) are to illustrate and confirm the analytical results obtained in the previous sections. The others, from (5) to (10), are numerical simulations for the games with large $d$ and $n$. They also  provide numerical support for Conjecture \ref{con: conjecture}.
\begin{enumerate}[(1)]
\item $f_d(t)$ and $f_d(t)/(d-1)$ as functions of $t$ for different values of $d$, see Figures \ref{fig:fd}a-b. Figure \ref{fig:fd}a illustrates the fourth property in Proposition \ref{pro: properties of f} that $f_d(t)$ is increasing as a function of $t$.  Figure \ref{fig:fd}b explains the scaling by showing that $f_d(t)/(d-1)$ is a bounded and decreasing function.
\item $f_d(t)$ and $f_d(t)/(d-1)$ as functions of $d$ for different values of $t$, see Figures \ref{fig:fd}c-d. These figures show that $f_d(t)$ increases with $d$ while $f_d(t)/(d-1)$ decreases, which are in agreement  with Conjecture \ref{cojecture: fd increases} and Theorem \ref{theo: fd/d decreases}. 
\item $f_d(x)$ and $f_d(x)/(d-1)$ as functions of the frequency $x$.
\item $\frac{\ln E(2,d)}{\ln (d-1)}$ as a function of $d$, see Figure \ref{fig:lnEdivlnD}a. This figure demonstrates Theorem \ref{theo: behavior of E2} on the convergence of $\frac{\ln E(2,d)}{\ln (d-1)}$ to $1/2$ as $d$ tends to infinity.
\item $f_{3,d}(t_1,t_2)$ as a function of $\mathbf{t}=(t_1,t_2)$ for different values of $d$, see Figures \ref{fig:f_n3n4}a-c. 
\item $f_{3,d}(t_1,t_2)$ as a function of $d$ for different values of $\mathbf{t}=(t_1,t_2)$, see Figure \ref{fig:lnfd34_depend_d}a.
\item $f_{4,d}(t_1,t_2,t_3)$ as a function of $\mathbf{t}=(t_1,t_2,t_3)$ for different values of $d$, see Figures \ref{fig:f_n3n4}d-f.
\item  $f_{4,d}(t_1,t_2,t_3)$ as a function of $d$ for different values of $\mathbf{t}=(t_1,t_2,t_3)$, see  Figure \ref{fig:lnfd34_depend_d}b.
\item $\frac{\ln E(3,d)}{\ln (d-1)}$ as a function of $d$, see  Figure \ref{fig:lnEdivlnD}b.
\item $\frac{\ln E(4,d)}{\ln (d-1)}$ as a function of $d$, see Figure \ref{fig:lnEdivlnD}c.
\end{enumerate}

In Figures  \ref{fig:f_n3n4}, we provide numerical results of  $f_{n,d}(\mathbf{t})$  for $n = 3$ and $n = 4$. We observe that the density function decreases  with $t_i$ (namely, $t_1$ and $t_2$ for $n = 3$, and $t_1$, $t_2$ and $t_3$ for $n = 4$) and increases with $d$. We conjecture  that for the general $d$-player $n$-strategy normal evolutionary game, the density function decreases  with $t_i$ and increases with $d$. 


Figures \ref{fig:lnEdivlnD}b and \ref{fig:lnEdivlnD}c support Conjecture \ref{con: conjecture}. From these two figures, one also can see the complexity of the problem when $d$ increases. We are able to run simulations for $d$ up to $10 000$ for $n=2$, up to $400$ for $n=3$ and only up to $20$ for $n=4$.

\begin{figure}[ht]
\centering
\includegraphics[width = 0.8\linewidth]{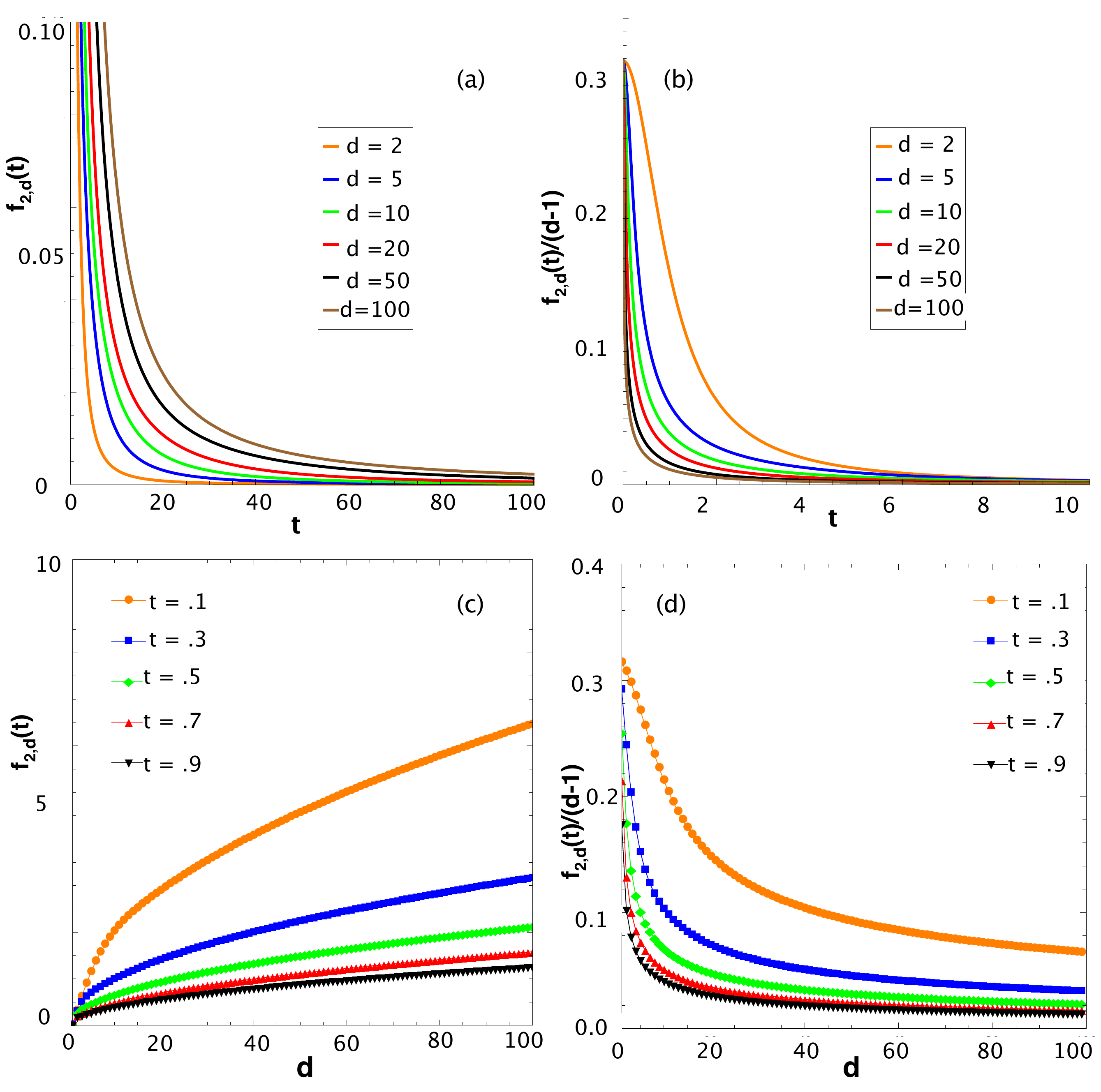}
\caption{\textbf{(a) Plot of $f_{2,d}(t)$ and (b)  of $f_{2,d}(t)/(d-1)$ as functions of $t$, for different values of $d$}. We observe that both functions  decrease with $t$, which are in accordance with Proposition \ref{pro: properties of f}. The second function is bounded from above, having maximum at $t = 0$, which agrees with Theorem \ref{theo: concentration}.   \textbf{(c) Plot of $f_{2,d}(t)$ and (d)  of $f_{2,d}(t)/(d-1)$ as functions of $d$, for different values of $t$}. We observe that $f_{2,d}(t)$ increases with $d$, while  $f_{2,d}(t)/(d-1)$ decreases, which are in agreement with Conjecture \ref{cojecture: fd increases} and Theorem \ref{theo: fd/d decreases}.   All the results were obtained  numerically  using Mathematica. }
\label{fig:fd}
\end{figure} 
\begin{figure}[ht]
\centering
\includegraphics[width =\linewidth]{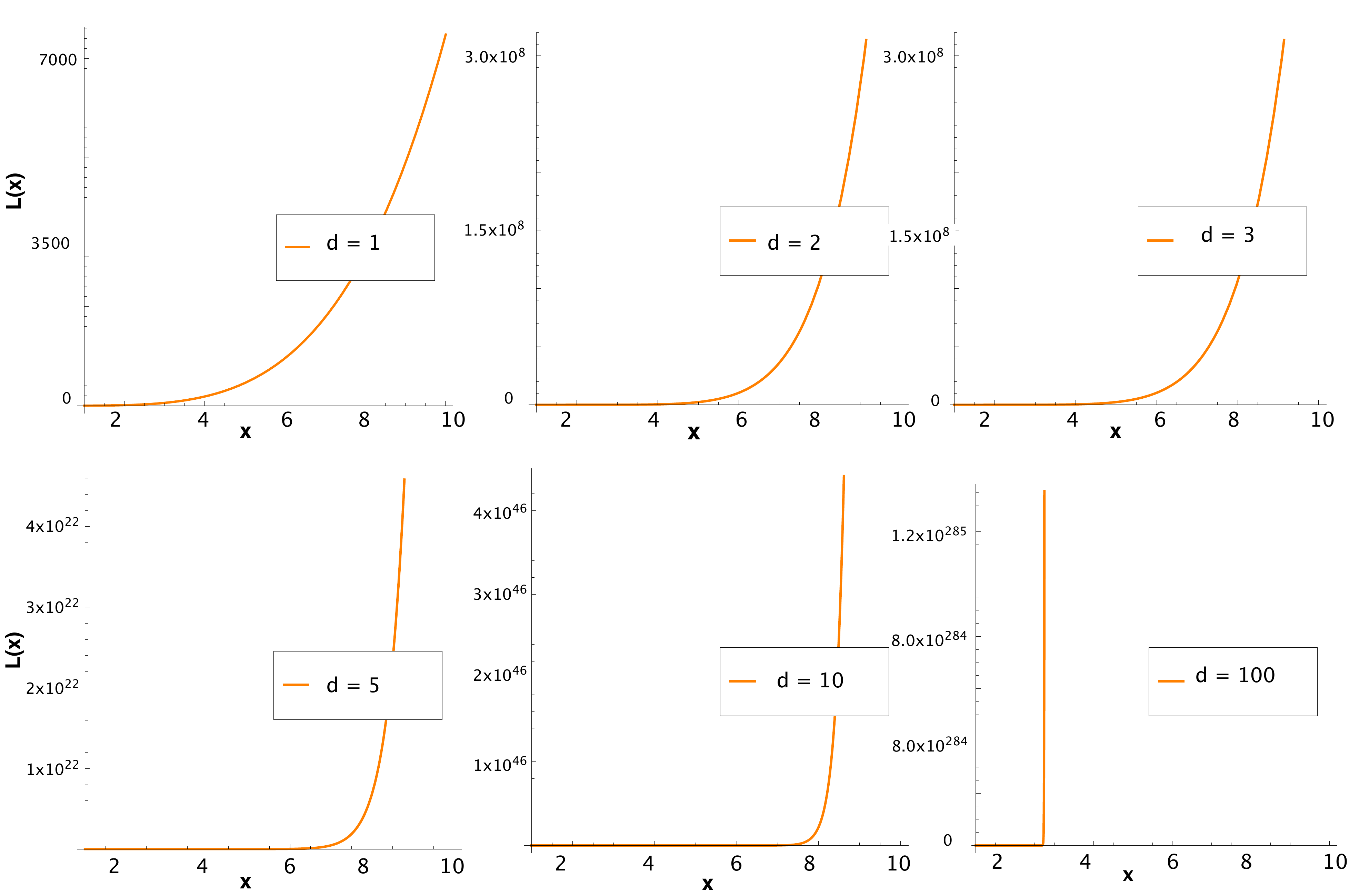}
\caption{\textbf{Plot of $L_d(x):=(2d+1)P_d^4 \red{-}   P_{d-1}^2 \left[ (2d-1) P_{d+1}^2 + 2 P_d^2  \right] 
$} for different values of $d$. We observe that $L_d(x)$ is always non-negative, thereby  supporting Conjecture \ref{cojecture: fd increases}, since the inequality  $L_d(x) \geq 0$ is the sufficient condition for the conjecture, as shown in Proposition \ref{prop: condition for f increase}}.
\label{fig:fd increase}
\end{figure} 
%

\begin{figure}[ht]
\centering
\includegraphics[width = \linewidth]{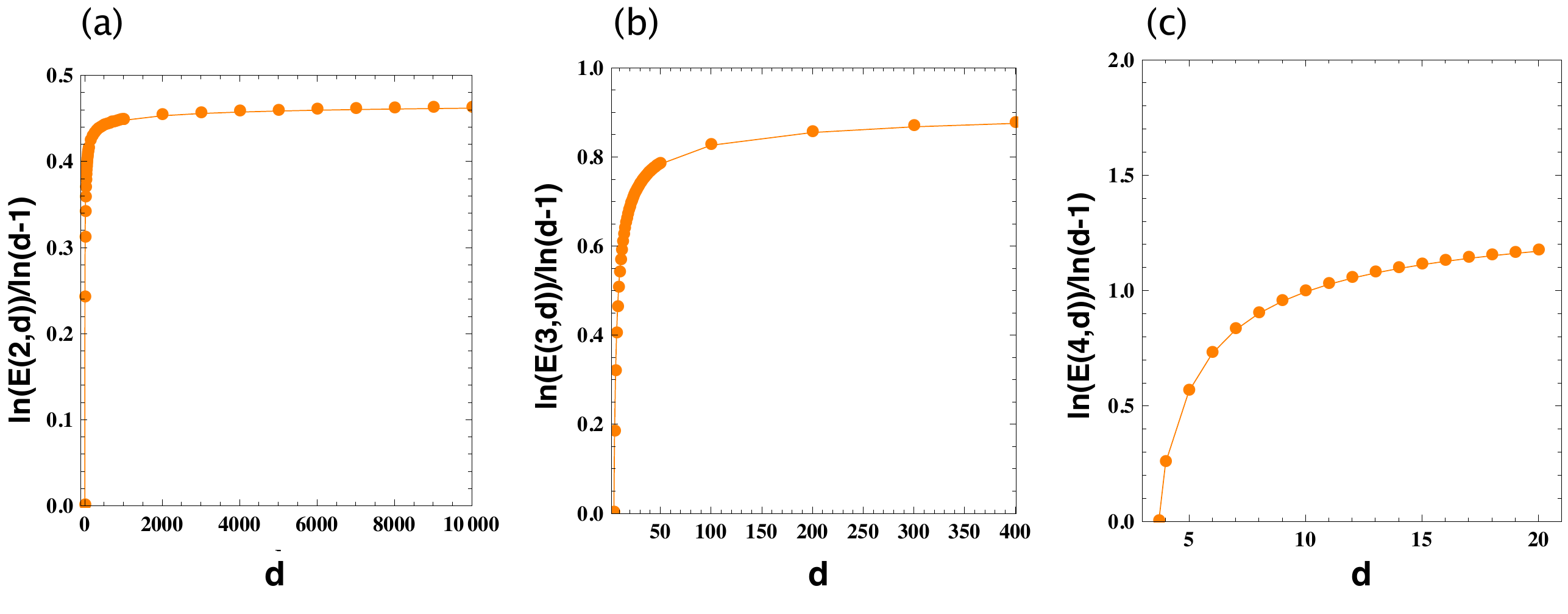}
\caption{\textbf{Plot of $\frac{\ln E(n,d)}{\ln (d-1)}$  as functions of $d$, (a) for $n = 2$, (b) for $n = 3$ and (c) for $n = 4$}. The results confirm the asymptotic behaviour for $n = 2$ as described  in Theorem \ref{theo: behavior of E2},  and clearly support Conjecture \ref{con: conjecture} for the general case.   All the results were obtained  numerically  using Mathematica. }
\label{fig:lnEdivlnD}
\end{figure} 

\begin{figure}[ht]
\centering
\includegraphics[width = \linewidth]{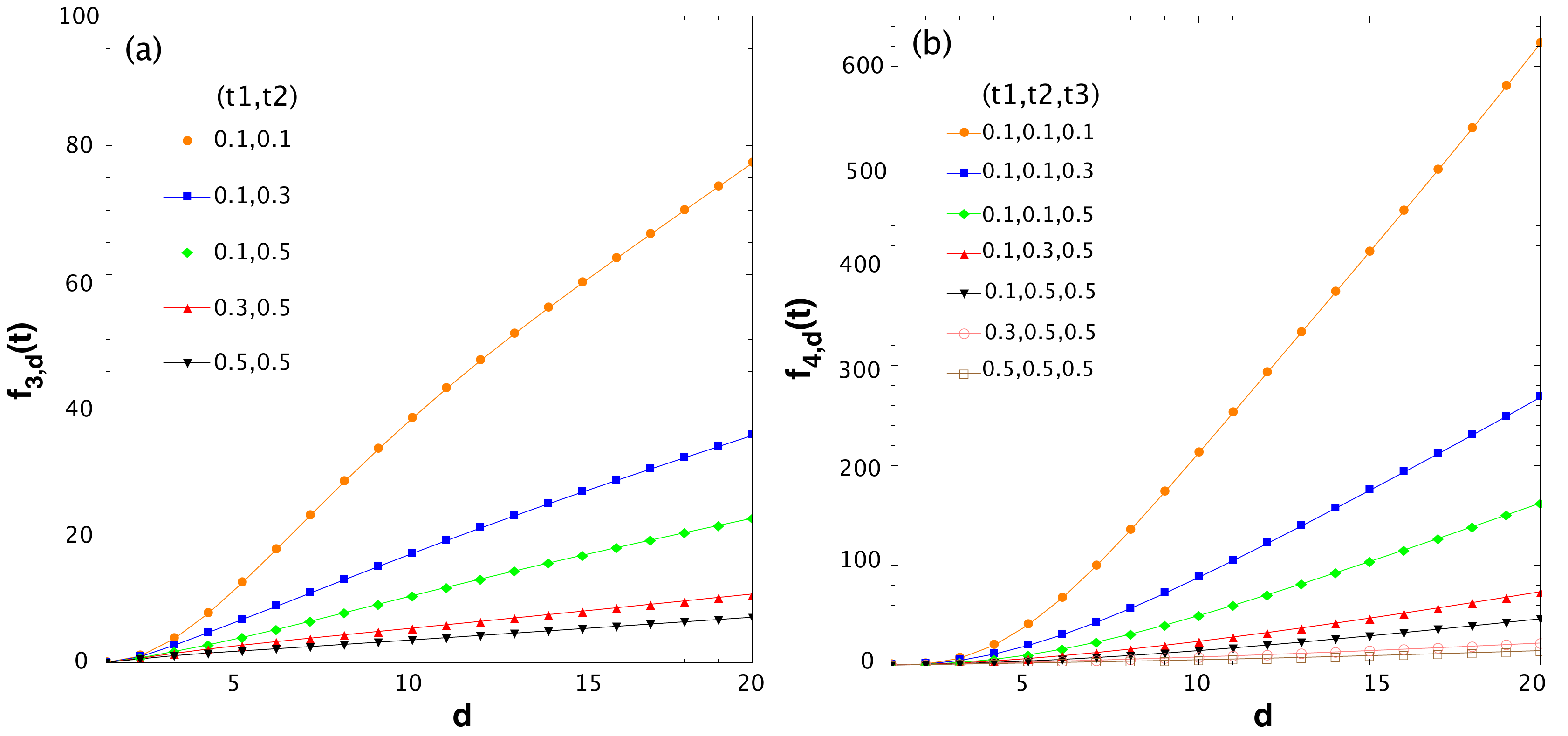}
\caption{\textbf{Plot of (a)  $f_{3,d}(t_1,t_2)$ for different values of $(t_1,t_2)$ and (b)  $f_{4,d}(t_1,t_2)$  for different values of $(t_1,t_2,t_3)$, both  as a function of $d$}. We observe  that the functions increase with $d$ while decrease with $t_i$, with $i = 1,2$ in  panel  (a) and $i = 1,2,3$ in panel (b).   All the results were obtained  numerically  using Mathematica. }
\label{fig:lnfd34_depend_d}
\end{figure} 
%

\begin{figure}[ht]
\centering
\includegraphics[width = \linewidth]{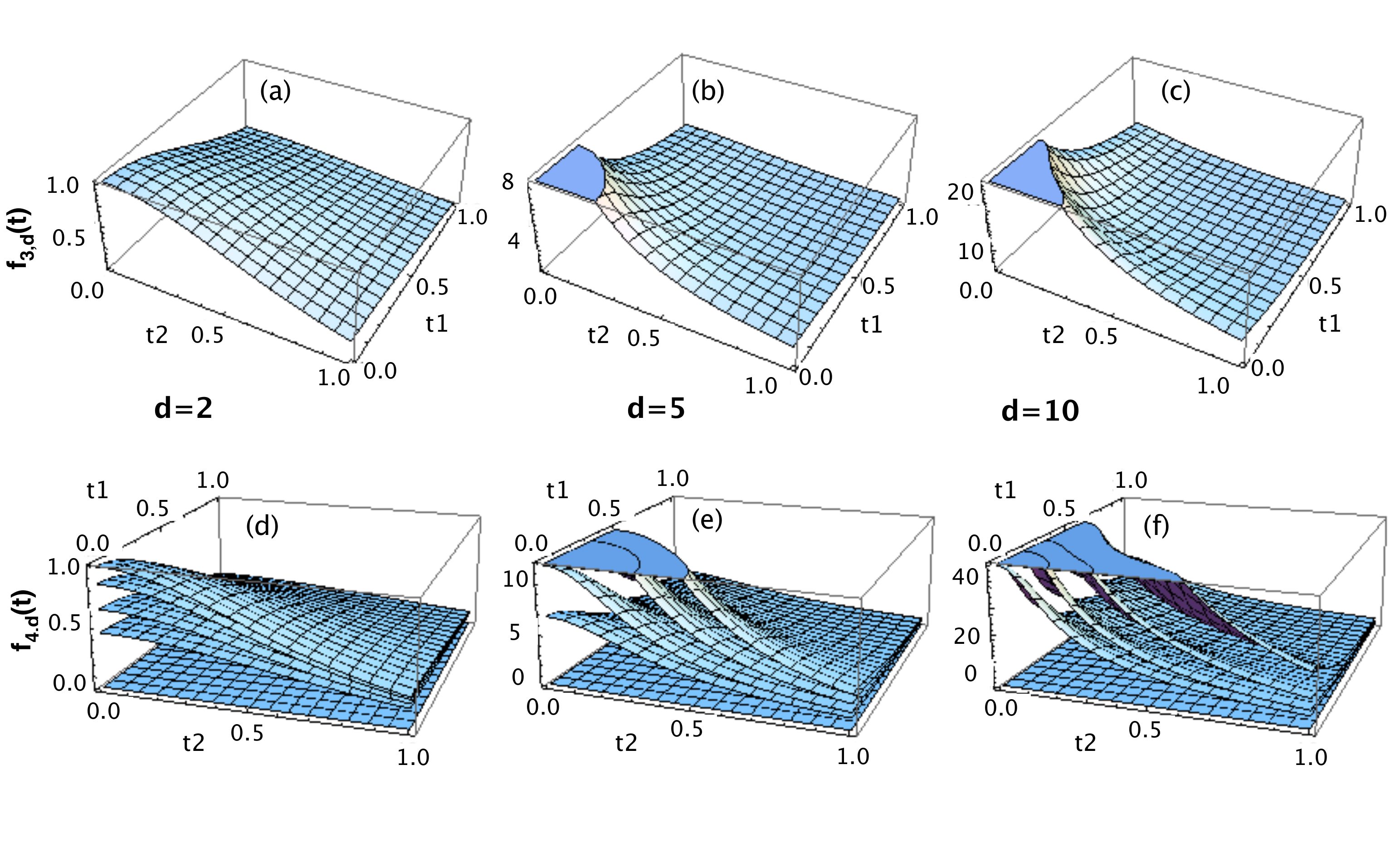}
\caption{\textbf{Plot of (a-c)  $f_{3,d}(t_1,t_2)$ as a function  of $(t_1,t_2)$; and (d-f)  $f_{4,d}(t_1,t_2)$  as a function of $(t_1,t_2)$, for different values of $t_3$}.  We observe  that both functions increase with $d$ while decrease with $t_i$, with $i = 1,2$ in  panels  (a)-(c) and $i = 1,2,3$ in panels (d)-(f).   In panels (a) and (d), $d = 2$; in panels (b) and (e), $d = 5$; in panels (c) and (f), $d = 10$. In panels (d)-(f), the surfaces, from bottom to top, correspond to $t_3 = 0.1, \ 0.3, \ 0.5, \ 0.7$ and $10$, respectively.  All the results were obtained  numerically  using Mathematica. }
\label{fig:f_n3n4}
\end{figure}

\begin{figure}[ht]
\centering
\includegraphics[width = \linewidth]{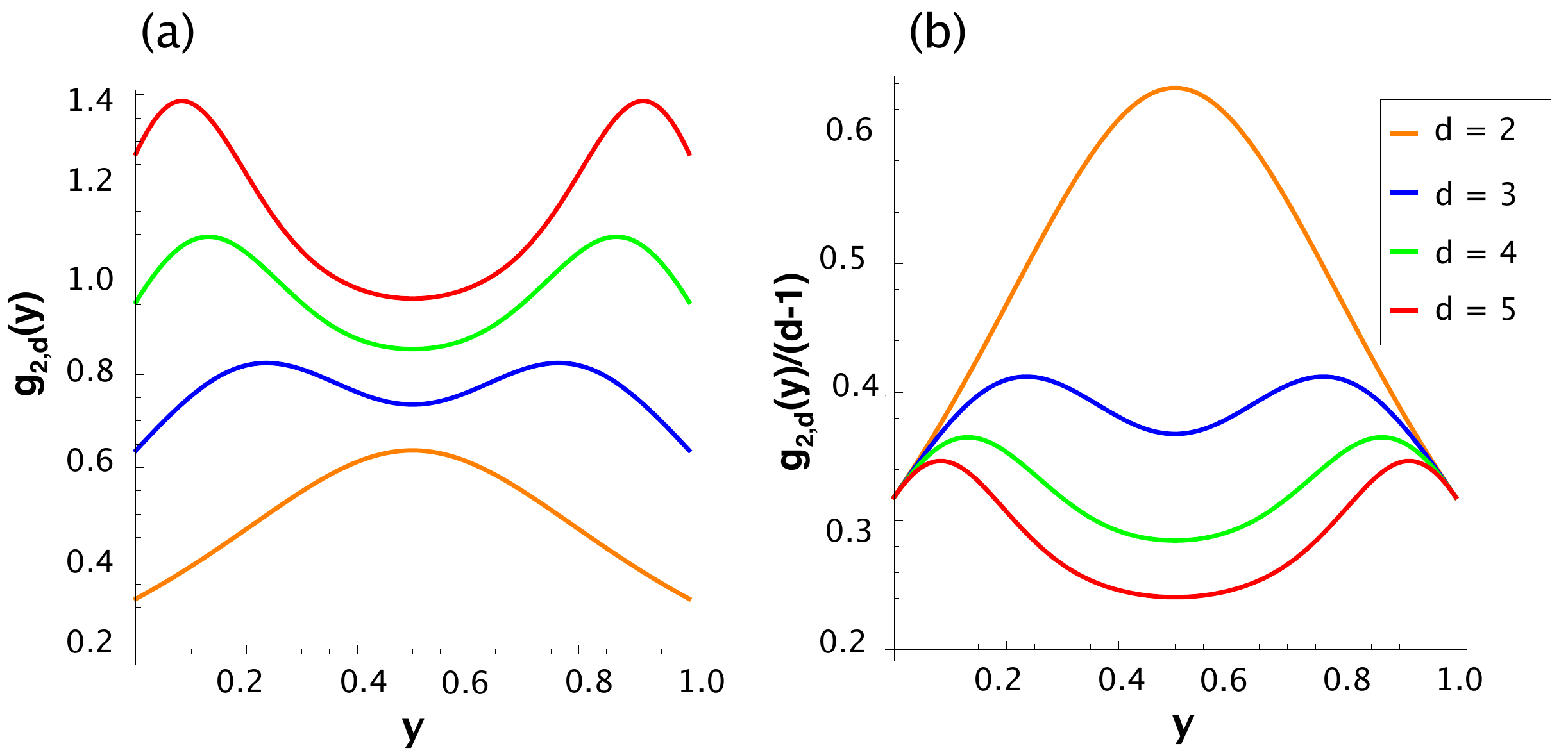}
\caption{\textbf{(a) Plot of $g_{2,d}(y)$ and (b)  of $g_{2,d}(y)/(d-1)$ as functions of $y$ (note that $t = \frac{y}{1-y}$), for different values of $d$}. We observe  that both functions    are symmetric about  the line $y = 1/2$, in accordance with Lemma \ref{lemma:symmetry of g}. The first function increases with $d$ while the second one decreases. All the results were obtained  numerically  using Mathematica.} 
\label{fig:f2d_depend_frequency_x}
\end{figure}

\begin{figure}[ht]
\centering
\includegraphics[width = \linewidth]{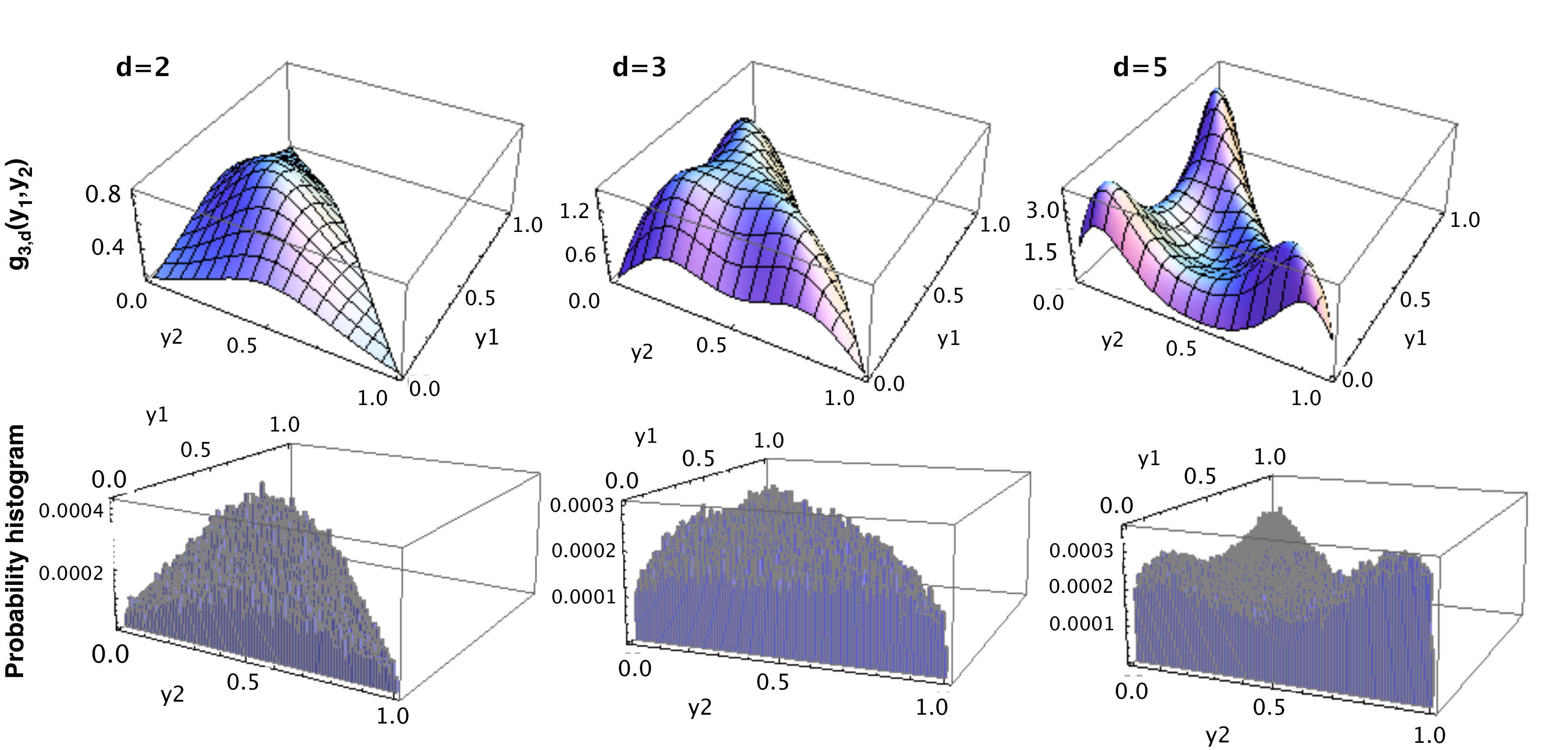}
\caption{\textbf{ (Top row)}  Plot of the density function $g_{3,d}(y_1, y_2)$ for different values of $d$ (here $y_1$ and $y_2$ are fractions of strategy 1 and 2, respectively, and $1-y_1-y_2$ is that of the third strategy). \textbf{(Bottom row)} Probability histograms plots of results from solving the system of equations  \ref{eq: eqn for fitness2} with $n = 3$ and for different $d$, where the payoff entries $\alpha$'s are (independently) sampled from Gaussian distribution with mean 0 and standard deviation of 0.5 ($10^6$ samples in total). The probability histograms have similar shapes to $g_{3,d}(y_1, y_2)$ (see also Table  \ref{table:Ed}).  All the results were obtained  numerically  using Mathematica.} 
\label{fig:theory vs samplings}
\end{figure}

\section{Discussion and outlook}
\label{sec: conclusion}
How do  equilibrium points in a general evolutionary game distribute if the payoff matrix entries are randomly drawn, and furthermore, how do they behave when the numbers of players and strategies change?
To address  these important questions regarding generic properties of general evolutionary  games, we have  analyzed here the density function, $f_{n,d}$,  and the expected number of (stable)  equilibrium points, $E(n,d)$ (respectively,  $\mathit{SE}(n,d)$), in a normal  $d$-player $n$-strategy  evolutionary  game. 
We have shown, analytically and using numerical simulations, that  $f_{2,d}(t)$  monotonically decreases with $t$  while it increases  with $d$.  The latter implies that, as the 
number of players in the game increases, it is more likely to see an equilibrium at a given point $t$. We also proved that its scaling with respect to the number of players in a game $d$, $\frac{f_{2,d}(t)}{d}$, decreases with $d$. 
More interestingly, we proved that this  density function asymptotically behaves in  the same manner  as $\sqrt{d-1}$ at any given   $t$ (i.e. regardless of the equilibrium point). 
 Similar monotonicity behaviors of the density function were observed numerically for games with larger numbers of strategies $n$. Additionally, we proved an upper bound for two-player game with an arbitrary number of strategies,  $n$. 
Briefly, our analysis of the expected density of equilibrium points has offered clear and fresh understanding about equilibrium distribution. Related to this distribution analysis, there have been some works analyzing patterns of ESSs and their attainability in concrete games, i.e. with a fixed payoff matrix  \cite{vickers1988patterns,cannings1988patterns,broom1994sequential}. Differently, our analysis addresses distribution of the general equilibrium points, and for generic games. Note also that those works  dealt with two-player games while we address here the general case (i.e. with arbitrary $d$). 
   
 
Regarding the expected  numbers of internal (stable) equilibrium points, first of all, as a result of  the described monotonicity properties of the density function, we established   analytically  that  $E(2,d)$ and $\mathit{SE}(2,d)$  increase with $d$ while their scaled forms, $\frac{E(2,d)}{d}$ and $\frac{\mathit{SE}(2,d)}{d}$, decrease with $d$.  Next, we proved a new upper bounds for $n = 2$,  with arbitrary  $d$: $E(2,d)\lesssim \sqrt{d-1}\ln(d-1)$. This upper bound is sharper than the one described in \cite{DH15} (which is also the only known one, to the best of our knowledge). As a consequence, a sharper upper bound for the number of expected number of stable equilibria can be established:  $\mathit{SE}(2,d)\lesssim \frac{1}{2}\sqrt{d-1}\ln(d-1)$.
More  importantly, that allowed us to derive  close-form limiting behaviors for such numbers:  $\lim\limits_{d\rightarrow\infty}\frac{\ln E(2,d)}{\ln(d-1)}=\frac{1}{2}$ and  
$\lim\limits_{d\rightarrow\infty}\frac{\ln \mathit{SE}(2,d)}{\ln(d-1)}=\frac{1}{2}$.  As such, apart from the mathematical elegance of our results, the analysis here has significantly improved  existing results on the studies of random evolutionary game theory   \cite{DH15,HTG12,gokhale:2010pn}. Moreover, generalizing these formulas for the general case, i.e. when the number of strategies $n$ is also arbitrary,  our  conjecture  that, $\lim\limits_{d\rightarrow\infty}\frac{\ln E(n,d)}{\ln(d-1)}=\frac{n-1}{2}$,  is nicely corroborated  by numerical simulations. 
Related to this analysis, there have been some works analyzing ESS  in  random evolutionary games \cite{haigh1988distribution,kontogiannis2009support,hart2008evolutionarily}. In both \cite{haigh1988distribution} and \cite{kontogiannis2009support}, the authors focused on the asymptotic behaviour of ESS with large  support sizes, i.e. considering also equilibrium points which are not internal, in random pairwise games. In a similar context, in  \cite{hart2008evolutionarily}, the authors studied ESS but with  support size of two, showing the asymptotic behaviors of such ESS when the number of strategies  varies. 
Differently from all these works (which  dealt exclusively with two-player interactions), our analysis copes with multi-player  games. Hence, our results have led to further understanding with respect to the  asymptotic behaviour of the expected number of stable equilibria for  multi-player random games.

Last but not least, as the density functions we analyzed here are closely related to Legendre polynomials, and actually, they are of the same form as the Bernstein polynomials, we have made a clear contribution to the longstanding  theory  of random polynomials  \cite{BP32,EK95,TaoVu14,NNV15TMP}(see again discussion in introduction).  
 We have derived asymptotic behaviors for the expected real zeros $E_{\B}$ of a random Bernstein polynomial, which, to our knowledge, had not been provided before. 
 Note that the  asymptotic behaviors and close forms of  the expected real zeros of some other well-known polynomials have been derived by other authors.  
For instance, the  Weyl polynomials for $a_i\colonequals\frac{1}{i!}$; the elliptic polynomials or binomial polynomials with $a_i\colonequals\sqrt{\begin{pmatrix}
d-1\\
i
\end{pmatrix}}$; and Kac polynomials with $a_i\colonequals 1$.
The difference between the polynomial studied  here (i.e. Bernstein) with the elliptic case: $a_i=\begin{pmatrix}
d-1\\
i
\end{pmatrix}$ are binomial coefficients, not their square root. In the elliptic case, $v(x)^TCv(y)=\sum_{i=1}^{d-1}\begin{pmatrix}
d-1\\
i
\end{pmatrix}x^iy^i=(1+xy)^{d-1}$, and as a result, $E(2,d)=\sqrt{d-1}$; While in our case, because of the square in the coefficients,
$v(x)^TCv(y)=\sum_{i=1}^{d-1}\begin{pmatrix}
d-1\\
i
\end{pmatrix}^2x^iy^i$, is no longer a generating function. 
Indeed, due to this difficulty, the analysis of the Bernstein polynomials is still rather limited, see for instance \cite{armentano2009note} for a detailed discussion. 

In short, our analysis has provided new  understanding about the generic behaviors of equilibrium  points in a general evolutionary game, namely, how they distribute and change in number  when the  number of players and that of the strategies in the game, are magnified.
\section{Appendix}
Detailed proofs of some lemmas and theorems in the previous sections are presented in this appendix.
\subsection{Proof of Lemma \ref{lem: connection btw Md and Pd}}
\label{app: lem: connection btw Md and Pd}
This relation has appeared in \cite[Exercise 101, Chapter 5]{GKP94}. For the readers' convenience, we provide a proof here.
\begin{proof}[Proof of Lemma \ref{lem: connection btw Md and Pd}]
Let $x=\frac{1+q}{1-q}$, from \eqref{eq: explicit reprentation of Legendre} we have
\begin{align*}
P_d\left(\frac{1+q}{1-q}\right)&=\frac{1}{2^d}\sum_{i=0}^d\begin{pmatrix}
d\\
i
\end{pmatrix}^2\left(\frac{1+q}{1-q}-1\right)^{d-i}\left(\frac{1+q}{1-q}+1\right)^i
\\&=\frac{1}{2^d}\sum_{i=0}^d\begin{pmatrix}
d\\
i
\end{pmatrix}^2\left(\frac{2q}{1-q}\right)^{d-i}\left(\frac{2}{1-q}\right)^i
\\&=\frac{1}{(1-q)^d}\sum_{i=0}^d\begin{pmatrix}
d\\
i
\end{pmatrix}^2q^{d-i}
\\&=\frac{1}{(1-q)^d}\sum_{i=0}^d\begin{pmatrix}
d\\
i
\end{pmatrix}^2q^{i}.
\end{align*}
Therefore,
\begin{equation*}
\sum_{i=0}^d\begin{pmatrix}
d\\
i
\end{pmatrix}^2q^{i}=(1-q)^d P_d\left(\frac{1+q}{1-q}\right).
\end{equation*}
By taking $q=t^2$, we obtain \eqref{eq: connection btw Md and Pd}.
\end{proof}
\subsection{Proof of Theorem \ref{theo: fd interm of Pd} }
\label{app: proof of Theorem fd interm of Pd}
\begin{proof}[Proof of Theorem \ref{theo: fd interm of Pd} ]

By taking the derivative of both sides in \eqref{eq: connection btw Md and Pd}, we obtain
\begin{align*}
M_{d+1}'(t)&=-2\,t\,d\,(1-t^2)^{d-1}P_d\left(\frac{1+t^2}{1-t^2}\right)+4t(1-t^2)^{d-2}P_d'\left(\frac{1+t^2}{1-t^2}\right)
\\&=-2\,t\,d\,\frac{M_{d+1}(t)}{1-t^2}+4t(1-t^2)^{d-2}P_d'\left(\frac{1+t^2}{1-t^2}\right).
\end{align*}
It follows that
\begin{equation*}
\frac{M_{d+1}'(t)}{M_{d+1}(t)}=\frac{-2\,t\,d}{1-t^2}+\frac{4t}{(1-t^2)^2}\frac{P_d'}{P_d}\left(\frac{1+t^2}{1-t^2}\right).
\end{equation*}
Now we compute the expression inside the square-root of the right-hand side of \eqref{eq: fd interms of Md}. We have
\begin{equation*}
t\frac{M_{d+1}'(t)}{M_{d+1}(t)}=\frac{-2\,t^2\,d}{1-t^2}+\frac{4t^2}{(1-t^2)^2}\frac{P_d'}{P_d}\left(\frac{1+t^2}{1-t^2}\right),
\end{equation*}
and
\begin{equation*}
\left(t\frac{M_{d+1}'(t)}{M_{d+1}(t)}\right)'=-\frac{4\,t\,d}{(1-t^2)^2}+\frac{8t(1+t^2)}{(1-t^2)^3}\frac{P_d'}{P_d}\left(\frac{1+t^2}{1-t^2}\right)+\frac{16t^3}{(1-t^2)^4}\frac{P_d''\,P_d-(P'_d)^2}{P_d^2}\left(\frac{1+t^2}{1-t^2}\right).
\end{equation*}
Substituting this expression into \eqref{eq: fd interms of Pd}, we get
\begin{equation}
\label{eq: fd interm of Pd 2}
(2\pi f_{d+1}(t))^2=-\frac{4\,d}{(1-t^2)^2}+\frac{8(1+t^2)}{(1-t^2)^3}\frac{P_d'}{P_d}\left(\frac{1+t^2}{1-t^2}\right)+\frac{16t^2}{(1-t^2)^4}\frac{P_d''\,P_d-(P'_d)^2}{P_d^2}\left(\frac{1+t^2}{1-t^2}\right).
\end{equation}
According to \eqref{eq: Legendre equation}, the Legendre polynomial $P_d$ satisfies the following equation for all $x\in \R$
\begin{equation*}
-2x P_d'(x)+(1-x^2)P''_d(x)=-d(d+1)P_d(x).
\end{equation*}
As a consequence, we obtain
\begin{equation*}
\frac{P''_d(x)}{P_d(x)}=\frac{1}{1-x^2}\left(2x\frac{P_d'(x)}{P_d(x)}-d(d+1)\right).
\end{equation*}
Substituting this expression into \eqref{eq: fd interm of Pd 2} with $x=\frac{1+t^2}{1-t^2}$, we get
\begin{align*}
(2\pi f_{d+1}(t))^2&=-\frac{4\,d}{(1-t^2)^2}+\frac{8(1+t^2)}{(1-t^2)^3}\frac{P_d'}{P_d}\left(\frac{1+t^2}{1-t^2}\right)
\\&\qquad+\frac{16t^2}{(1-t^2)^4}\left\{\frac{1}{1-\left(\frac{1+t^2}{1-t^2}\right)^2}\left[2\frac{1+t^2}{1-t^2}\frac{P_d'}{P_d}\left(\frac{1+t^2}{1-t^2}\right)-d(d+1)\right]-\left(\frac{P_d'}{P_d}\right)^2\left(\frac{1+t^2}{1-t^2}\right)\right\}
\\&=\frac{4d^2}{(1-t^2)^2}-\frac{16t^2}{(1-t^2)^4}\left(\frac{P_d'}{P_d}\right)^2\left(\frac{1+t^2}{1-t^2}\right),
\end{align*}
which is the claimed relation \eqref{eq: fd interms of Pd}.
\end{proof}
\subsection{Proof of Theorem \ref{theo: fd interms of Pd and Pd-1}}
\label{app: proof of Theorem fd interms of Pd and Pd-1}
\begin{proof}[Proof of Theorem \ref{theo: fd interms of Pd and Pd-1}]
Using the following relation of the Legendre polynomials for all $x\in\R$
\begin{equation*}
P_d'(x)=\frac{d}{x^2-1}\left(xP_d(x)-P_{d-1}(x)\right),
\end{equation*}
we get
\begin{equation}
\label{eq: P'd/Pd}
\frac{P_d'(x)}{P_d(x)}=\frac{d}{x^2-1}\left(x-\frac{P_{d-1}(x)}{P_d(x)}\right).
\end{equation}
In particular, taking $x=\frac{1+t^2}{1-t^2}$, we obtain
\begin{align*}
\frac{P'_d}{P_d}\left(\frac{1+t^2}{1-t^2}\right)&=\frac{d}{\left(\frac{1+t^2}{1-t^2}\right)^2-1}\left[\frac{1+t^2}{1-t^2}-\frac{P_{d-1}}{P_d}\left(\frac{1+t^2}{1-t^2}\right)\right]
\\&=\frac{d(1-t^2)^2}{4t^2}\left[\frac{1+t^2}{1-t^2}-\frac{P_{d-1}}{P_d}\left(\frac{1+t^2}{1-t^2}\right)\right].
\end{align*}
Substituting this expression into \eqref{eq: fd interms of Pd}, we achieve
\begin{equation*}
(2\pi f_{d+1}(t))^2=\frac{4d^2}{(1-t^2)^2}-\frac{d^2}{t^2}\left[\frac{1+t^2}{1-t^2}-\frac{P_{d-1}}{P_d}\left(\frac{1+t^2}{1-t^2}\right)\right]^2,
\end{equation*} 
which is \eqref{eq: main fd interms of Pd}.
\end{proof}
\subsection{Proof of Lemma \ref{lem: Turan inequality}}
\label{app: proof of lemma Turan inequality}
\begin{proof}[Proof of Lemma \ref{lem: Turan inequality}]
This lemma follows directly from \cite{constantinescu2005inequality} (Theorem 2.1) where the authors proved that 
$$P_d(x)^2 -  P_{d+1}(x)P_{d-1}(x) = \frac{1-x^2}{d (d+1)} \left( \sum_{i=1}^d \frac{1}{i} + \sum_{i=1}^{d-1} \frac{1}{i+1}  \sum_{j=1}^i (2j+1) P_j^2(x)\right), $$
which is negative for all  $|x| \geq 1$.
\end{proof}
\subsection{Proof of Proposition \ref{prop: condition for f increase}}
\label{app: proof of lemma condition for f increase}
\begin{proof}[Proof of Proposition \ref{prop: condition for f increase}]
We will prove that
\begin{equation}
f^2_{d+2}(t)-f^2_{d+1}(t)\geq 0 \quad\Longleftrightarrow\quad \eqref{eq: condition for fd increase}.
\end{equation}
From \eqref{eq: main fd interms of Pd}, we have
\begin{align*}
4\pi^2(f^2_{d+2}(t)-f^2_{d+1}(t))&=\frac{4(d+1)^2}{(1-t^2)^2}-\frac{(d+1)^2}{t^2}\left[\frac{1+t^2}{1-t^2}-\frac{P_{d}}{P_{d+1}}\left(\frac{1+t^2}{1-t^2}\right)\right]^2
\\&\qquad-\frac{4d^2}{(1-t^2)^2}+\frac{d^2}{t^2}\left[\frac{1+t^2}{1-t^2}-\frac{P_{d-1}}{P_d}\left(\frac{1+t^2}{1-t^2}\right)\right]^2
\\&=\frac{4(2d+1)}{(1-t^2)^2}-\frac{1}{t^2}\left[\frac{1+t^2}{1-t^2}-\frac{(d+1)P_d^2-dP_{d-1}P_{d+1}}{P_dP_{d+1}}\left(\frac{1+t^2}{1-t^2}\right)\right]\times
\\&\qquad\left[(2d+1)\frac{1+t^2}{1-t^2}-\frac{(d+1)P_d^2+dP_{d-1}P_{d+1}}{P_dP_{d+1}}\left(\frac{1+t^2}{1-t^2}\right)\right].
\end{align*}
Therefore $f_d(t)$ is increasing as a function of $d$ if and only if 
\begin{align*}
&\left[\frac{1+t^2}{1-t^2}-\frac{(d+1)P_d^2-dP_{d-1}P_{d+1}}{P_dP_{d+1}}\left(\frac{1+t^2}{1-t^2}\right)\right]\times\left[(2d+1)\frac{1+t^2}{1-t^2}-\frac{(d+1)P_d^2+dP_{d-1}P_{d+1}}{P_dP_{d+1}}\left(\frac{1+t^2}{1-t^2}\right)\right]
\\&\qquad\leq \frac{4(2d+1)t^2}{(1-t^2)^2}.
\end{align*}
We re-write the expression above using the variable $x$, using the relation $x^2-1=\frac{4t^2}{(1-t^2)^2}$, as follows
\begin{equation}
\label{eq: ineq}
\left[x-\frac{(d+1)P_d^2-dP_{d-1}P_{d+1}}{P_dP_{d+1}}(x)\right]\times\left[(2d+1)x-\frac{(d+1)P_d^2+dP_{d-1}P_{d+1}}{P_dP_{d+1}}(x)\right]\leq (2d+1)(x^2-1).
\end{equation}
We now simplify this  expression  using the recursion relation of the Legendre polynomials, i.e.  $dP_{d-1}=(2d+1)xP_d-(d+1)P_{d+1}$. Namely, we have
\begin{align*}
xP_dP_{d+1}-(d+1)P_d^2+dP_{d-1}P_{d+1}&=xP_dP_{d+1}-(d+1)P_d^2+[(2d+1)xP_d-(d+1)P_{d+1}]P_{d+1}
\\&=(d+1)[-P_d^2-P_{d+1}^2+2xP_dP_{d+1}]
\end{align*}
and 
\begin{align*}
&(2d+1)xP_dP_{d+1}-(d+1)P_d^2-dP_{d-1}P_{d+1}
\\&\qquad=(2d+1)xP_dP_{d+1}-(d+1)P_d^2-[(2d+1)xP_d-(d+1)P_{d+1}]P_{d+1}
\\&\qquad=(d+1)(P^2_{d+1}-P_d^2).
\end{align*}
Substituting these calculations into \eqref{eq: ineq} we obtain \eqref{eq: condition for fd increase}.

To prove the second assertion of Proposition \ref{prop: condition for f increase}, we proceed as follows.
Let 
\begin{align*}
H_{d+1}  & = (d+1)^2\left[P^2_{d+1}(x)+P^2_{d}(x)-2x P_{d+1}(x)P_d(x)\right] \\ 
&\qquad= \left[(2d+1)x P_{d}(x)- d P_{d-1}(x)\right]^2 +   (d+1)^2 P^2_{d}(x)
\\&\qquad\qquad-2x (d+1)  \left[(2d+1)x P_{d}(x)- d P_{d-1}(x)\right]P_d(x) \\
&\qquad=\left[(2d+1)^2 x^2 + (d+1)^2 - 2(d+1)(2d+1) x^2  \right]  P^2_{d}(x)
\\&\qquad\qquad-\left[2d(2d+1)x  - 2d(d+1) x  \right] P_{d}(x) P_{d-1}(x)
+ d^2 P^2_{d-1}(x)   \\
&\qquad = \left[d^2 + (2d+1)(1-x^2) \right] P^2_{d}(x)+d^2 P^2_{d-1}(x)-2x d^2 P_{d}(x)P_{d-1}(x)  \\
&\qquad= d^2 \left[P^2_{d}(x)+P^2_{d-1}(x)-2x P_{d}(x)P_{d-1}(x)\right]   + (2d+1)(1-x^2)  P^2_{d}(x) \\
&\qquad= H_d + (2d+1)(1-x^2)  P^2_{d}(x).
\end{align*}
Hence,  the expression in \eqref{eq: condition for fd increase} can be simplified as follows 
\begin{align*}
&H_{d+1} [P^2_{d+1}(x)-P^2_{d}(x)] +(2d+1)(x^2-1)P^2_d(x)P^2_{d+1}(x) \\
 &= [H_d + (2d+1)(1-x^2)  P^2_{d}(x)][P^2_{d+1}(x)-P^2_{d}(x)] +(2d+1)(x^2-1)P^2_d(x)P^2_{d+1}(x) \\
   &= H_d [P^2_{d+1}(x)-P^2_{d}(x)] +(2d+1)(x^2-1)P^4_d(x) \\
     &= \frac{P^2_{d+1}(x)-P^2_{d}(x)}{P^2_{d}(x)-P^2_{d-1}(x)}\left[H_d  \left(P^2_{d}(x)-P^2_{d-1}(x)\right) +(2d-1)(x^2-1)P^2_d(x)P^2_{d-1}(x) \cdot Q \right],
%
%
\end{align*}
where $Q  = \frac{(2d+1)P^2_{d}(P^2_{d}-P^2_{d-1})}{(2d-1)P^2_{d-1}(P^2_{d+1}-P^2_{d})} $. 
Suppose that \eqref{eq: condition for fd increase 2} is true, i.e., 
\begin{equation*} 
(2d+1)P_d^4 \geq   P_{d-1}^2 \left[ (2d-1) P_{d+1}^2 + 2 P_d^2  \right].
\end{equation*}
This implies that $Q\geq 1$ for all $x$ and $d$. Then it follows that
\begin{align}
\label{eq: H inequality}
&H_{d+1} [P^2_{d+1}(x)-P^2_{d}(x)] +(2d+1)(x^2-1)P^2_d(x)P^2_{d+1}(x)\nonumber
\\&\qquad\geq \frac{P^2_{d+1}(x)-P^2_{d}(x)}{P^2_{d}(x)-P^2_{d-1}(x)}\left[H_d  \left(P^2_{d}(x)-P^2_{d-1}(x)\right) +(2d-1)(x^2-1)P^2_d(x)P^2_{d-1}(x)\right]\nonumber
\\&\qquad\geq  \frac{P^2_{d+1}(x)-P^2_{d}(x)}{P^2_{d}(x)-P^2_{d-1}(x)}\times \frac{P^2_{d}(x)-P^2_{d-1}(x)}{P^2_{d-1}(x)-P^2_{d-2}(x)}\nonumber
\\&\qquad\qquad\times \left[H_{d-1}  \left(P^2_{d-1}(x)-P^2_{d-2}(x)\right) +(2d-3)(x^2-1)P^2_{d-1}(x)P^2_{d-2}(x)\right]\nonumber
\\&\qquad\geq\cdots\nonumber
\\&\qquad\geq \prod\limits_{i=1}^{d}\frac{P^2_{i+1}(x)-P^2_{i}(x)}{P^2_{i}(x)-P^2_{i-1}(x)}\times\left[H_1  \left(P^2_{1}(x)-P^2_{0}(x)\right) +(x^2-1)P^2_{1}(x)P^2_{0}(x)\right].
\end{align}
By definition of $H_d$, we have
\begin{equation*}
H_1=P_1^2(x)+P^2_0(x)-2x P_{1}(x)P_0(x)=x^2+1-2x^2=1-x^2.
\end{equation*}
Substituting this into \eqref{eq: H inequality}, we obtain
\begin{align*}
&H_{d+1} [P^2_{d+1}(x)-P^2_{d}(x)] +(2d+1)(x^2-1)P^2_d(x)P^2_{d+1}(x)
\\&\geq (x^2-1)\prod\limits_{i=1}^{d}\frac{P^2_{i+1}(x)-P^2_{i}(x)}{P^2_{i}(x)-P^2_{i-1}(x)}\\&=P^2_{d+1}(x)-P_d^2(x)\geq 0,
\end{align*}
i.e., the condition \eqref{eq: condition for fd increase} is satisfied.
\end{proof}

\begin{acknowledgements}
We would like to thank the anonymous referee for his/her useful suggestions that helped us to improve the presentation of the manuscript. 
\end{acknowledgements}



\end{document}